\documentclass[a4paper,reqno,12pt]{amsart}

\usepackage{color,amsmath,amssymb,hyperref,mathtools}

\usepackage{tikz}
\usetikzlibrary{topaths}
\usepackage[margin=2.4cm]{geometry}

\usepackage{amsthm}
\makeatletter
\def\th@plain{%
  \thm@notefont{}% same as heading font
  \itshape % body font
}
\def\th@definition{%
  \thm@notefont{}% same as heading font
  \normalfont % body font
}
\makeatother
\theoremstyle{plain}
\newtheorem{theorem}{Theorem}[section]
\newtheorem{lemma}[theorem]{Lemma}
\newtheorem{proposition}[theorem]{Proposition}
\newtheorem{corollary}[theorem]{Corollary}
\newtheorem*{corollary*}{Corollary}

\theoremstyle{definition}
\newtheorem{definition}[theorem]{Definition}
\newtheorem{remark}[theorem]{Remark}

\newcommand{\Q}{\mathbb{Q}}

\newcommand{\N}{\mathbb{N}}
\newcommand{\I}{^{-1}}
\DeclareMathOperator{\Aut}{Aut}

\DeclareMathOperator{\lk}{lk}
\DeclareMathOperator{\st}{st}
\DeclareMathOperator{\dlk}{{\lk}{\downarrow}}
\DeclareMathOperator{\dst}{{\st}{\downarrow}}
\DeclareMathOperator{\BU}{BU}
\DeclareMathOperator{\DBU}{{BU}{\downarrow}}
\DeclareMathOperator{\SBU}{SBU}
\DeclareMathOperator{\Sid}{{\Sigma}{\downarrow}}

\newcommand{\SAut}{\Sigma\!\Aut}  %no space between Sigma and Aut

\newcommand{\defeq}{\mathbin{\vcentcolon =}}

\numberwithin{equation}{section}

\begin{document}

\title[Stability for partially symmetric automorphism groups]{Rational homological stability for groups of partially symmetric automorphisms of free groups}

\author{Matthew C. B. Zaremsky}
\address{Fakult\"at f\"ur Mathematik \\
Universit\"at Bielefeld \\
Bielefeld, Germany 33615}
\email{zaremsky@math.uni-bielefeld.de}

\date{\today}
%geometric gp thry; aut gps of gps, top methods in gp thry
\thanks{The author was supported by the SFB~$701$ of the DFG}

\begin{abstract}
Let $F_{n+m}$ be the free group of rank $n+m$, with generators $x_1,\dots,x_{n+m}$. An automorphism $\phi$ of $F_{n+m}$ is called partially symmetric if for each $1\le i\le m$, $\phi(x_i)$ is conjugate to $x_j$ or $x_j\I$ for some $1\le j\le m$. Let~$\SAut_n^m$ be the group of partially symmetric automorphisms. We prove that for any $m\ge 0$ the inclusion $\SAut_n^m\to\SAut_{n+1}^m$ induces an isomorphism in rational homology for dimensions~$i$ satisfying $n\ge (3(i+1)+m)/2$, with a similar statement for the groups $P\SAut_n^m$ of pure partially symmetric automorphisms. We also prove that for any $n\ge 0$ the inclusion $\SAut_n^m\to\SAut_n^{m+1}$ induces an isomorphism in rational homology for dimensions~$i$ satisfying $m>(3i-1)/2$.
\end{abstract}

\maketitle

\section{Introduction}
\label{sec:intro}

Let $\Aut(F_{n+m})$ be the group of automorphisms of the free group $F_{n+m}$. For a fixed basis~$\{x_1,\dots,x_{n+m}\}$ of $F_{n+m}$, an automorphism~$\phi$ of $F_{n+m}$ is called \emph{partially symmetric} if for each $1\le i\le m$, $\phi(x_i)$ is conjugate to~$x_j$ or~$x_j\I$ for some~$1\le j\le m$. If $\phi$ is an automorphism such that each $\phi(x_i)$ is even conjugate to~$x_i$ we call~$\phi$ \emph{pure partially symmetric}. Call these first~$m$ generators \emph{distinguished} and the other~$n$ \emph{undistinguished}. Let~$\SAut_n^m$ be the group of partially symmetric automorphisms of $F_{n+m}$, and $P\SAut_n^m$ the group of pure partially symmetric automorphisms.

We prove that the rational homology of these groups is stable in the parameters~$n$ and~$m$, and the rational homology of~$\SAut_n^m$ is also stable in~$m$.  This means that the rational homology is independent of the parameters once they are large enough.  This question was posed in \cite{mcewen_thesis}, where a general strategy was outlined, involving a hypothetical Morse function on a space introduced in \cite{bux09}.  As a first step, in \cite{mcewen_thesis,zar_degree_thm} a Morse function was constructed for the spine of Auter space, which provided a simplified proof of the so called Degree Theorem of \cite{hatcher98}. From the Degree Theorem, the rational homological stability of $\Aut(F_n)=\SAut_n^0$ can be deduced. With this Morse-theoretic approach in hand for the classical case, it was supposed that one should then be able to generalize the situation to~$\SAut_n^m$, but this was left in the conjectural stage in \cite{mcewen_thesis}. In the present work we complete this project; namely, we exhibit a Morse function that yields a generalized version of the Degree Theorem, from which we deduce rational homological stability for $\SAut_n^m$.

To keep the notation straight, we mention that in \cite{bux09} the ``outer'' version of the group we are calling $P\SAut_n^m$ is denoted $P\Sigma(n,k)$, where~$n$ is the rank and~$k$ the number of distinguished generators. In \cite{jensen04} the same group is denoted $A_n^k$, where~$n$ and~$k$ are the number of undistinguished and distinguished generators, respectively. In \cite{jensen04} certain other groups denoted $A_{n,k}$ are considered, which are central extensions of~$A_n^k$, but these are not the same as the groups~$\SAut_n^m$ considered here. For example, the automorphisms that properly permute the distinguished generators of $F_{n+m}$ appear only in $\SAut_n^m$, and not in $P\SAut_n^m=A_n^m$ or in $A_{n,m}$.

\medskip

The relevant existing results are as follows. In \cite{hatcher98} it is shown that the homology of $\Aut(F_n)=\SAut_n^0$ is stable with respect to~$n$. In \cite[Corollary~1.2]{galatius11} the stable rational homology is even shown to be trivial, namely, $H_i(\Aut(F_n);\Q)=0$ for all $n>2i+1$. At the other end of the spectrum, in \cite{hatcher10} it is shown that the group of \emph{symmetric automorphisms} $\SAut(F_m)=\SAut_0^m$ is homologically stable in~$m$, and it turns out the rational homology actually vanishes in every dimension \cite{griffin12,wilson12}. In contrast, the pure case is quite different. The rational homology of~$P\SAut_0^m$ is not stable in~$m$ \cite{jensen04}, and in fact the cohomology ring has been completely computed \cite{jensen06}. To use the notation of \cite{jensen04}, so $P\SAut_0^m$ is denoted $A_n^m$, while the $A_0^m$ are not homologically stable, the groups $A_{n,m}$ are in fact stable in $n$ and $m$ \cite{hatcher05}. We remark that the methods used to prove stability for $A_{n,m}$ are very different from how we will prove stability for $\SAut_n^m$ here.

We actually obtain stability results for a range of families of subgroups of~$\SAut_n^m$, which includes the groups $P\SAut_n^m$. Consider any family of groups~$G_n^m$ such that
$$P\SAut_n^m\le G_n^m\le\SAut_n^m$$
for each~$n$ and~$m$, and such that the inclusion
$$\SAut_n^m\hookrightarrow \SAut_{n+1}^m,$$
given by extending $\phi\in\SAut_n^m$ to $F_{m+n+1}$ via $\phi(x_{n+m+1})=x_{n+m+1}$, restricts to an inclusion $G_n^m\hookrightarrow G_{n+1}^m$. Of course $P\SAut_n^m$ and~$\SAut_n^m$ themselves are examples of such families of groups. Our main result for these groups is the following theorem.

\begin{theorem}[Stability in $n$]\label{hom_stab_thm}
 For any $m\ge 0$ and $i\ge 0$, and any family of groups~$G_n^m$ satisfying the above conditions, the map
$$H_i(G_n^m;\Q)\to H_i(G_{n+1}^m;\Q)$$
induced by inclusion is an isomorphism for $n\ge (3(i+1)+m)/2$.
\end{theorem}

\begin{corollary*}
 The rational homology of~$\SAut_n^m$ is stable in~$n$, as is the rational homology of $P\SAut_n^m$.\qed
\end{corollary*}

We also consider stability in the other parameter,~$m$. Renumber the elements of the basis by $\{x_1,\dots,x_n,x_{n+1},\dots,x_{n+m}\}$, so an automorphism $\phi$ is partially symmetric if for all $1\le i\le m$, $\phi(x_{n+i})$ is conjugate to $x_{n+j}$ or $x_{n+j}\I$ for some $1\le j\le m$. We now have a natural inclusion map
$$\SAut_n^m\hookrightarrow \SAut_n^{m+1},$$
given by extending $\phi\in\SAut_n^m$ to $F_{n+m+1}$ via $\phi(x_{n+m+1})=x_{n+m+1}$.

\begin{theorem}[Stability in $m$]\label{hom_stab_thm2}
 For any $n\ge 0$ and $i\ge 0$, the map
$$H_i(\SAut_n^m;\Q)\to H_i(\SAut_n^{m+1};\Q)$$
induced by inclusion is an isomorphism for $m>(3i-1)/2$.
\end{theorem}

In Section~\ref{sec:auterspace} we provide some background on the spine of Auter space $K_{n+m}$, and describe a subcomplex~$\nabla K_n^m$ that admits a nice~$\SAut_n^m$ action. We also filtrate~$\nabla K_n^m$ using the notion of \emph{weighted degree}, a generalization of \emph{degree} from \cite{hatcher98}. In Section~\ref{sec:morse} we define a height function~$h$ on~$\nabla K_n^m$, which generalizes the height function from the classical case, constructed in \cite{zar_degree_thm}. In Section~\ref{sec:hom_stab} we show how the Generalized Degree Theorem, Theorem~\ref{thrm:sublev_conn}, yields our homological stability results, and in Section~\ref{sec:connectivity} we prove the Generalized Degree Theorem. To prove this, we show that descending links with respect to our height function~$h$ are highly connected. This is done by separately considering two join factors, the \emph{down-link}, in Section~\ref{sec:down_conn}, and the \emph{up-link}, in Section~\ref{sec:up_conn}.

\subsection*{Acknowledgments} The author is grateful to Kai-Uwe Bux for his guidance in this project, and to James Griffin, Rob McEwen and Jenny Wilson for their helpful advice along the way.

\section{Auter space and our space of interest}\label{sec:auterspace}

We will analyze the homology of~$\SAut_n^m$ by considering its action on a certain simplicial complex. Our starting point is the well-studied \emph{spine of Auter space}~$K_n$ introduced in \cite{hatcher98}. Let~$R_n$ be the rose with~$n$ edges, i.e., the graph with a single vertex~$p_0$ and~$n$ edges. Here by a \emph{graph} we mean a connected one-dimensional CW-complex, with the usual notions of vertices and edges. We identify~$F_n$ with $\pi_1(R_n)$. If~$\Gamma$ is a graph with basepoint vertex~$p$, a homotopy equivalence $\rho \colon R_n \to \Gamma$ is called a \emph{marking} on~$\Gamma$ if~$\rho$ takes~$p_0$ to~$p$. We will consider two markings to be equivalent if there is a basepoint-preserving homotopy between them. Also, we only consider graphs such that~$p$ is at least bivalent and all other vertices are at least trivalent. Note that we do allow \emph{separating edges}, that is edges whose complement in the graph is disconnected.

For graphs~$\Gamma_1$ and~$\Gamma_2$, a basepoint-preserving homotopy equivalence $d \colon \Gamma_1 \to \Gamma_2$ is called a \emph{forest collapse} or a \emph{blow-down} if it amounts to collapsing a subforest of~$\Gamma_1$. The reverse of a blow-down is, naturally, called a \emph{blow-up}. This gives us a partial ordering on equivalence classes of triples $(\Gamma,p,\rho)$, namely $(\Gamma',p,\rho') \le (\Gamma,p,\rho)$ if there is a forest collapse $d \colon \Gamma \to \Gamma'$ such that~$\rho'$ is equivalent to $d\circ\rho$. The spine $K_n$ of Auter space is then the geometric realization of the poset of equivalence classes of triples $(\Gamma,p,\rho)$ with~$\Gamma$ a rank~$n$ graph, with this partial ordering. In particular the vertices of~$K_n$ are equivalence classes of marked basepointed graphs.

Since we are identifying~$F_n$ with $\pi_1(R_n)$, we can also identify $\Aut(F_n)$ with the group of basepoint-preserving homotopy equivalences of~$R_n$, up to homotopy. This is of course the same as the group of markings of~$R_n$, so we can denote markings on~$R_n$ by elements of $\Aut(F_n)$. There is a (right) action of $\Aut(F_n)$ on $K_n$ in the following way: given a vertex $(\Gamma,p,\rho)$ in~$K_n$ and $\phi\in\Aut(F_n)$, we have
$$\phi(\Gamma,p,\rho)=(\Gamma,p,\rho\circ\phi).$$
In particular this action only affects markings, and it is easy to see that $\Aut(F_n)$ permutes markings arbitrarily.

\noindent\textbf{Viable marked graphs:} To analyze the groups~$\SAut_n^m$ we first pass to a certain (full) subcomplex~$\Delta K_n^m$ of~$K_{n+m}$. The vertices of~$\Delta K_n^m$ are marked basepointed graphs $(\Gamma,p,\rho)$ such that~$\Gamma$ is a \emph{viable graph} and~$\rho$ is an \emph{admissible marking}. A graph is viable if it contains~$m$ reduced cycles $C_1,\dots,C_m$ in~$\Gamma$ that are pairwise disjoint. See Figure~\ref{fig:viable_graph} for an example. A marking~$\rho$ is called admissible if there is a maximal tree~$T$ in~$\Gamma$ such that for $\pi \colon \Gamma \to \Gamma/T = R_{n+m}$, we have $\pi\circ\rho\in \SAut_n^m$ (recall our identification of $\Aut(F_{n+m})$ with markings on~$R_{n+m}$), and the reduced cycles~$C_i$ obtained by reducing $\rho(x_i)$ for $1\le i\le m$ are pairwise disjoint. See \cite{bux09,mcewen_thesis} for more details. For brevity we will just define a \emph{viable marked graph} to be a viable graph with an admissible marking. The cycles~$C_i$ for $1\le i\le m$ are called \emph{distinguished cycles}, and we similarly refer to vertices, edges, and half-edges as \emph{distinguished} if they are contained in some~$C_i$. A forest~$F$ in a viable marked graph~$\Gamma$ is called \emph{admissible} if~$\Gamma/F$ is again viable and the induced marking is again admissible. The characterizing property of admissible forests is that any tree~$T$ in an admissible forest can intersect at most one distinguished cycle~$C$, and if~$T\cap C$ is nonempty then it must either be a single vertex or a connected edge path in~$C$. An example of an admissible and an inadmissible forest (for some marking~$\rho$) are shown in gray in Figure~\ref{fig:viable_graph}.

\begin{figure}[htb]
    \centering
    \begin{tikzpicture}%viable graph
  \draw[line width=0.7pt]
   (-0.8,0.85) -- (0,0) -- (1.6,0.5)
   (-0.6,1.1) -- (0.2,0.5) -- (0,0) -- (1.6,1.5) -- (0.2,0.5)
   (1.05,1.8) -- (1.6,1.5)
   (-0.6,1.3) -- (0.3,1.9)
   (-0.25,1.55) -- (0.6,0.8)
   (0.3,1.9) to [out=175, in=75,looseness=1] (-0.9,1.6)
   (0,0) to [out=140, in=20,looseness=1] (-1,0.4)
   (-1,0.4) to [out=250, in=160,looseness=0.5] (0,0);
  \filldraw
   (0,0) circle (1.5pt);
  \draw[line width=2.3pt]
   (2,0.6) circle (0.4 cm)
   (1.9,1.8) circle (0.4 cm)
   (0.7,2) circle (0.4 cm)
   (-1,1.2) circle (0.4 cm);

  \begin{scope}[xshift=4.1cm]%admissible forest
  \draw[line width=0.7pt]
   (0,0) -- (1.6,0.5)
   (-0.25,1.55) -- (0.6,0.8)
   (1.6,1.5) -- (0.2,0.5);
  \draw[line width=2pt, gray]
   (0,0) -- (1.6,0.5)
   (-0.25,1.55) -- (0.6,0.8)
   (1.6,1.5) -- (0.2,0.5);
  \draw[line width=0.7pt]
   (-0.8,0.85) -- (0,0)
   (-0.6,1.1) -- (0.2,0.5) -- (0,0) -- (1.6,1.5)
   (1.05,1.8) -- (1.6,1.5)
   (-0.6,1.3) -- (0.3,1.9)
   (0.3,1.9) to [out=175, in=75,looseness=1] (-0.9,1.6)
   (0,0) to [out=140, in=20,looseness=1] (-1,0.4)
   (-1,0.4) to [out=250, in=160,looseness=0.5] (0,0);
  \filldraw
   (0,0) circle (1.5pt);
  \draw[line width=2.3pt]
   (2,0.6) circle (0.4 cm)
   (1.9,1.8) circle (0.4 cm)
   (0.7,2) circle (0.4 cm)
   (-1,1.2) circle (0.4 cm);

  \end{scope}

  \begin{scope}[xshift=8.2cm]%non-admissible forest
  \draw[line width=0.7pt]
   (-0.8,0.85) -- (0,0) -- (1.6,0.5)
   (-0.6,1.1) -- (0.2,0.5) -- (0,0) -- (1.6,1.5) -- (0.2,0.5)
   (-0.25,1.55) -- (0.6,0.8)
   (0.3,1.9) to [out=175, in=75,looseness=1] (-0.9,1.6)
   (0,0) to [out=140, in=20,looseness=1] (-1,0.4)
   (-1,0.4) to [out=250, in=160,looseness=0.5] (0,0);
  \draw[line width=2pt, gray]
   (1.05,1.8) -- (1.6,1.5)
   (-0.6,1.3) -- (0.3,1.9);
  \filldraw
   (0,0) circle (1.5pt);
  \draw[line width=2.3pt]
   (2,0.6) circle (0.4 cm)
   (1.9,1.8) circle (0.4 cm)
   (0.7,2) circle (0.4 cm)
   (-1,1.2) circle (0.4 cm);

  \end{scope}

\end{tikzpicture}
    \caption{From left to right: A viable graph, an admissible forest and an inadmissible forest.}
    \label{fig:viable_graph}
\end{figure}
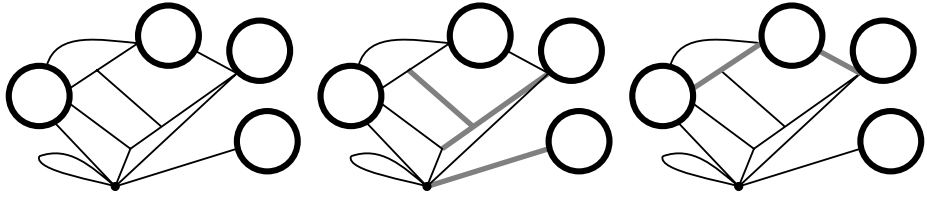
 
The action of~$\SAut_n^m$ on $K_{n+m}$ only affects markings, and takes admissible markings to admissible markings, so we can consider the action of~$\SAut_n^m$ on~$\Delta K_n^m$. Let
$$\Delta Q_n^m \defeq \Delta K_n^m/\SAut_n^m$$
be the orbit space.

\begin{proposition}\cite[Section~3]{bux09}\label{bcv_prop}
~$\SAut_n^m$ acts on~$\Delta K_n^m$ with finite stabilizers and finite quotient $\Delta Q_n^m$, and~$\Delta K_n^m$ is contractible.
\end{proposition}

It is also clear that if an element of~$\SAut_n^m$ stabilizes a simplex then it fixes it pointwise, since the vertices of any simplex correspond to pairwise non-isomorphic graphs. The upshot of this that $\Delta Q_n^m$ and~$\SAut_n^m$ have the same rational homology; see for example Exercise~2 on page~174 in \cite{brown82}.

\noindent\textbf{Our space of interest:} There is a nice subcomplex of~$\Delta K_n^m$ that will prove useful for our purposes, namely the subcomplex $\nabla K_n^m$ spanned by marked basepointed graphs in~$\Delta K_n^m$ in which the basepoint~$p$ is not contained in a distinguished cycle. The action of~$\SAut_n^m$ on $\nabla K_n^m$ similarly features finite stabilizers and finite quotient
$$\nabla Q_n^m \defeq \nabla K_n^m/\SAut_n^m.$$
To keep straight which is which, note that the symbol~$\nabla$ is ``top-heavy'' compared to $\Delta$, indicating that the distinguished cycles cannot be down at the basepoint.

\noindent\textbf{Weighted degree:} It is difficult to analyze $\Delta Q_n^m$ and $\nabla Q_n^m$ directly, and so we will work with a certain filtration. For a vertex $(\Gamma,p,\rho)$ in~$\Delta K_n^m$, define the \emph{weighted valency} $val_w(v)$ of a vertex~$v$ to be the number of undistinguished half-edges at~$v$, plus half the number of distinguished half-edges. Define the \emph{weighted degree} $d_w(\Gamma)$ to be
$$d_w(\Gamma)\defeq 2n+m-val_w(p).$$
It is clear that $1\le val_w(p)\le 2n+m$, and so $0\le d_w(\Gamma)\le N$, where $N \defeq 2n+m-1$. As an example, the reader can verify that the weighted degree of the graph in Figure~\ref{fig:viable_graph} is~$10$. We will also make use of the notion of \emph{degree} from \cite{hatcher98}, which we define to be
$$d_0(\Gamma) \defeq 2n+2m-val(p).$$
If~$c$ denotes the number of distinguished cycles not containing~$p$, it is clear that $d_w=d_0-c$. The reader curious about the motivation for defining weighted degree this way should glance ahead to the paragraph after Definition~\ref{def:lollies}.

For $k\in\N_0$ let $\Delta K_{n,k}^m$ be the full subcomplex of~$\Delta K_n^m$ spanned by marked basepointed graphs with weighted degree less than or equal to~$k$. In particular for~$k\ge N$, $\Delta K_{n,k}^m=\Delta K_n^m$. Also let $\nabla K_{n,k}^m = \Delta K_{n,k}^m \cap \nabla K_n^m$. The sequence of spaces
$$\nabla K_{n,0}^m \subseteq \nabla K_{n,1}^m \subseteq \cdots$$
is a filtration of $\nabla K_n^m$, and not of the contractible complex~$\Delta K_n^m$, but these smaller complexes will prove to be the right ones to inspect for various reasons. Note that when~$m=0$, $\nabla K_{n,k}^0=\Delta K_{n,k}^0=K_{n,k}$, the filtration of~$K_n$ by degree used in \cite{hatcher98}.

As a bit of foreshadowing to Section~\ref{sec:hom_stab}, note that the undistinguished loop and/or the distinguished loop on a stick (or ``lollipop'') at the basepoint in Figure~\ref{fig:viable_graph} could be removed without changing the weighted degree. This property is precisely the motivation for filtrating~$\Delta K_n^m$ and $\nabla K_n^m$ by weighted degree.

\section{A height function}\label{sec:morse}

We now define a height function~$h$ on the vertices of~$\Delta K_n^m$. This height function is related to the one defined in \cite{zar_degree_thm} on the space~$K_n=\Delta K_n^0$. This will allow us to inspect the connectivity of~$\nabla K_{n,k}^m$ using discrete Morse theory; see \cite{bestvina97} for background on discrete Morse theory.

\begin{definition}[Features of graphs]\label{def:graph_defs}
 Let $(\Gamma,p,\rho)$ be a basepointed viable marked graph. For vertices $v,v'$ in~$\Gamma$, let the \emph{distance}~$d(v,v')$ be the number of edges in a minimal-length edge path from~$v$ to~$v'$. Also, for a subforest~$F$ of~$\Gamma$, define the \emph{level}~$D(F)$ of~$F$ to be the smallest~$i$ such that~$F$ has a vertex at distance~$i$ from~$p$. Let
$$\Lambda_i(\Gamma)\defeq \{v\in\Gamma\mid d(p,v)=i\}$$
be the~$i^{\text{th}}$ \emph{level} of~$\Gamma$, so for example $\Lambda_0(\Gamma)=\{p\}$. If~$v$ is a vertex that is in a distinguished cycle $C$, and $d(p,v)\le d(p,v')$ for any other vertex~$v'$ in~$C$, then we will say that~$v$ is a \emph{base vertex for~$C$}, and call $i_C\defeq d(p,v)$ the \emph{base height} of~$C$. If~$v$ is a base vertex for some~$C$, call~$v$ a base vertex. Note that the basepoint~$p$ is a base vertex if and only if it is distinguished, if and only if~$c=m-1$. In Figure~\ref{fig:base_vtcs} the distinguished cycle~$C$ is indicated by thick edges, the base vertices are the larger dots, and the basepoint is the largest dot at the bottom.
\end{definition}

\begin{figure}[htb]
    \centering
    \begin{tikzpicture}%base vertices
  \draw[line width=0.7pt]
   (-2,1.5) -- (0,0)    (-1,1.5) -- (0,0)   (2,1.5) -- (0,0)   (0,1.5) -- (2,3)   (0,2.5) -- (2,3)   (-2,3) -- (-1,2.5)   (-1,1.5) -- (0,2.5)
   (0,0) to [out=30, in=-40,looseness=1.8] (2,3);
  
  \draw[line width=1.8pt]
   (-2,3) -- (2,3) -- (2,1.5) -- (0,1.5) -- (0,2.5) -- (-1,2.5) -- (-1,1.5) -- (-2,1.5) -- (-2,3);

  \filldraw
  (0,0) circle (4pt)

  (-1,1.5) circle (2.7pt)
  (-2,1.5) circle (2.7pt)
  (2,1.5) circle (2.7pt)
  (2,3) circle (2.7pt)

  (-2,3) circle (1.2pt)
  (-1,2.5) circle (1.2pt)
  (0,2.5) circle (1.2pt)
  (0,1.5) circle (1.2pt);
\end{tikzpicture}
    \caption{Distinguished cycle $C$ with $i_C=1$.}
    \label{fig:base_vtcs}
\end{figure}
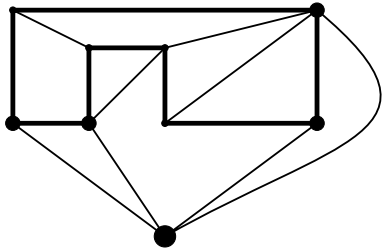

\noindent\textbf{Measurements contributing to the height function:} For each $i\ge 0$ let $m_i(\Gamma)$ be the number of base vertices in $\Lambda_i(\Gamma)$, define $n_i(\Gamma)\defeq -|\Lambda_i(\Gamma)|$ and let
$$d_i(\Gamma)\defeq \sum_{v\not\in\Lambda_i}(val(v)-2).$$
Note that $m_0=m-c$, $n_0$ is constant~$-1$ and $d_0=2n+2m-val(p)$ is the degree. In general~$d_i$ can be thought of as counting the number of vertices \emph{not} at level~$i$, with higher valent vertices ``counting for more.'' Now define
$$h_i(\Gamma)\defeq (m_i(\Gamma),n_i(\Gamma),d_i(\Gamma))\text{, and set }h(\Gamma)=(h_0(\Gamma),h_1(\Gamma),h_2(\Gamma),\dots)$$
with the lexicographic order. We remark that the height function used in \cite{zar_degree_thm} on the spine of Auter space was $(d_0,n_1,d_1,n_2,d_2,\dots)$, which encodes the same information as our~$h$ when~$m=0$. Extend~$h$ to the vertices of~$\Delta K_n^m$ via $h(\Gamma,p,\rho)=h(\Gamma)$. In general we will just write~$\Gamma$ to denote vertices of~$\Delta K_n^m$, with the basepoint and marking understood.

\noindent\textbf{How forests affect the measurements:} Note that for any admissible forest~$F$, blowing down~$F$ either increases or decreases $h_{D(F)}$. For example, if~$n_{D(F)}$ does not change, then~$d_{D(F)}$ must decrease. Also, blowing down~$F$ does not change any~$h_i$ for~$i<D(F)$, since this is clearly true for~$m_i$ and~$n_i$, and is easy to check for~$d_i$. In general, of all the terms changed by blowing down~$F$, there is one that is lexicographically first, which we will call the \emph{essential} term of~$F$. Similarly, any blow-up has an essential term. We remark that a blow-down at level~$i$ cannot decrease~$n_i$, and a blow-up at level~$i$ cannot decrease~$d_i$, though both blow-downs and blow-ups can either increase or decrease~$m_i$.

It is easy to see that $\nabla K_n^m$ is the sublevel set of~$\Delta K_n^m$ defined by the inequality
$$h \le (0,0,0,\dots).$$
Moreover, when~$m_0=0$ we have~$d_0=d_w+m$, so $\nabla K_{n,k}^m$ is the sublevel set of~$\Delta K_n^m$ defined by
$$h \le (0,-1,k+m+1,-1,0,\dots).$$
The upshot of this is that connectivity of $\nabla K_{n,k}^m$ can be determined by looking at descending links of vertices with respect to~$h$. For a vertex~$\Gamma$ in~$\Delta K_n^m$, the \emph{descending star}~$\dst(\Gamma)$ with respect to~$h$ is the set of simplices in the star of~$\Gamma$ whose other vertices all have strictly lower height than~$\Gamma$. The \emph{descending link}~$\dlk(\Gamma)$ consists of the faces of simplices in~$\dst(\Gamma)$ that do not themselves contain~$\Gamma$.

\noindent\textbf{The up-link and down-link:} There are two types of vertices in~$\dlk(\Gamma)$: those obtained from~$\Gamma$ by a descending blow-up, and those obtained by a descending blow-down. Here we say that a blow-up or blow-down is \emph{descending} if the resulting graph has a lower height than the starting graph. Call the subcomplex of~$\dlk(\Gamma)$ spanned by vertices of the first type the \emph{up-link}, and the subcomplex spanned by vertices of the second type the \emph{down-link}. Any vertex in the up-link is related to every vertex in the down-link by a blow-down, so~$\dlk(\Gamma)$ is the simplicial join of the up- and down-links. (This is exactly the kind of decomposition of~$\dlk(\Gamma)$ that occurs in \cite{zar_degree_thm}.) We remark that we only consider admissible blow-downs, and on the other hand observe that any blow-up of a viable graph is again viable. If a forest blow-down is descending we call the forest itself \emph{descending}, and similarly we refer to \emph{ascending} forests. As remarked above, any forest blow-down either increases or decreases $h$. Since adjacent vertices of~$\Delta K_n^m$ are related via forest blow-downs, this means that adjacent vertices have different heights, so~$h$ is really a height function, in the sense of~\cite{bestvina97}.

\medskip

It will be important to have a somewhat explicit description of which forests are descending.

\begin{lemma}[Interpreting the height function $h$]\label{which_forests_desc}
 Let~$F$ be an admissible forest in~$\Gamma$ with~$i \defeq D(F)$.
 \begin{enumerate}
  \item If $m_i(\Gamma/F)<m_i(\Gamma)$, then~$F$ is descending.
  \item If $m_i(\Gamma/F)>m_i(\Gamma)$, then~$F$ is ascending.
  \item If $m_i(\Gamma/F)=m_i(\Gamma)$ and~$F$ connects vertices in~$\Lambda_i$, then~$F$ is ascending.
  \item If $m_i(\Gamma/F)=m_i(\Gamma)$ and~$F$ does not connect vertices in~$\Lambda_i$, then~$F$ is descending.
 \end{enumerate}
\end{lemma}

\begin{proof}
 The essential term of~$F$ occurs in~$h_i$, so the first two claims are immediate. Suppose $m_i(\Gamma/F)=m_i(\Gamma)$. If~$F$ connects vertices in~$\Lambda_i$, then blowing down~$F$ increases~$n_i$ and so is ascending. If~$F$ does not connect vertices in~$\Lambda_i$, then blowing down~$F$ does not change~$n_i$, but decreases~$d_i$, so~$F$ is descending.
\end{proof}

The height function~$h$ is a bit cumbersome, but the idea of how it will be used is not too complicated. The goal is to prove the Generalized Degree Theorem, Theorem~\ref{thrm:sublev_conn}, that $\nabla K_{n,k}^m$ is $(k-1)$-connected, for $0 \le k< N$. Since~$\Delta K_n^m$ is contractible, it suffices by~\cite[Corollary~2.6]{bestvina97} to show that for any vertex~$\Gamma$ in $\Delta K_n^m \setminus \nabla K_{n,k}^m$, the descending link~$\dlk(\Gamma)$ of~$\Gamma$ in~$\Delta K_n^m$ is $(k-1)$-connected. We will do this in Section~\ref{sec:connectivity}. Before doing the technical connectivity calculations though, in Section~\ref{sec:hom_stab} we show how the Generalized Degree Theorem gives us homological stability results.

\section{Homological stability}\label{sec:hom_stab}
In the next section we will prove the Generalized Degree Theorem, namely that $\nabla K_{n,k}^m$ is $(k-1)$-connected for $0\le k<N$, where recall that $N\defeq 2n+m-1$. First, in this section, we show how this can be used to obtain homological stability results for families of groups. As in Section~\ref{sec:intro}, let~$G_n^m$ be any family of groups such that $P\SAut_n^m\le G_n^m\le\SAut_n^m$ for each~$n$ and~$m$, and such that the inclusion $\SAut_n^m\hookrightarrow \SAut_{n+1}^m$ restricts to an inclusion $G_n^m\hookrightarrow G_{n+1}^m$. For any~$0\le k< N$, the action of~$G_n^m$ on~$\Delta K_{n,k}^m$ has finite stabilizers and finite quotient $\nabla K_{n,k}^m/G_n^m$. Hence by the Generalized Degree Theorem, $\nabla K_{n,k}^m/G_n^m$ has the same rational homology as~$G_n^m$ in dimensions~$i$ with $i < k$. To be precise, we have the following

\begin{lemma}[From groups to orbit spaces]\label{gp_to_space_hlgy}
 For any $0\le k< N$, we have that $H_i(\nabla K_{n,k}^m/G_n^m;\Q)$ is isomorphic to $H_i(G_n^m;\Q)$ for $i<k$, and $H_k(\nabla K_{n,k}^m/G_n^m;\Q)$ surjects onto $H_k(G_n^m;\Q)$.\qed
\end{lemma}

To get homological stability in~$n$ for~$G_n^m$ we can now look for homological stability in~$n$ for $\nabla K_{n,k}^m/G_n^m$. We will do this in a similar way as done in the classical $m=0$ case in \cite[Section~5]{hatcher98}. The vertices of $\nabla K_n^m/P\SAut_n^m$ are the homeomorphism types of basepointed graphs with~$m$ distinguished oriented cycles, disjoint and distinguishable from each other and from the basepoint. In $\nabla K_{n,k}^m/\SAut_n^m$ the cycles become non-oriented and indistinguishable from each other, and in general $\nabla K_{n,k}^m/G_n^m$ interpolates between these two extremes. Exactly as in \cite{hatcher98}, we have a map
$$\nu \colon \nabla K_{n,k}^m/G_n^m \hookrightarrow \nabla K_{n+1,k}^m/G_{n+1}^m$$
induced by sending a graph~$\Gamma$ to~$\Gamma\vee S^1$, that is the graph with an extra (undistinguished) loop wedged to its basepoint.

To get stability in~$n$, we want to be able to ``detect'' loops and theta subgraphs at the basepoint. If~$\Gamma$ has a loop at the basepoint~$p$ then $\Gamma$ is in the image of~$\nu$, which is why want to be able to detect loops. We will see in Proposition~\ref{orbit_spaces_n_stable} why theta subgraphs at the basepoint are also useful.

First we set up the situation for stability in~$m$. Instead of loops and theta subgraphs we will use certain subgraphs defined as follows.

\begin{definition}[Lollipops and double lollipops]\label{def:lollies}
 A \emph{lollipop} in $\Gamma$ is a subgraph $\ell$ consisting of an undistinguished non-loop edge~$\epsilon$ (the \emph{stick}) and a distinguished loop $\delta$ sharing a vertex $v\neq p$, such that $\epsilon$ and $\delta$ are the only edges incident to $v$. If $w\neq v$ is the other vertex of $\epsilon$, we define a \emph{double lollipop} to be the result of wedging $\ell$ at $w$ to any point of another lollipop $\ell'$.
\end{definition}

Define a map
$$\mu \colon \nabla K_{n,k}^m/\SAut_n^m \hookrightarrow \nabla K_{n,k}^{m+1}/\SAut_n^{m+1}$$
by sending $\Gamma$ to $\Gamma\vee \ell$, where $\ell$ is a lollipop wedged to the basepoint. Unlike attaching an undistinguished loop, attaching a lollipop in this way changes the degree, but it does not change the weighted degree, so this is still fine. (Indeed this was precisely the impetus for defining weighted degree as we did.) We now describe how to detect the presence of these various subgraphs at the basepoint, as in \cite[Lemma~5.2]{hatcher98}. Following that, we will see why this gives us stability.

\begin{lemma}[Detecting features at the basepoint]\label{detect_things_at_p}
 Let $(\Gamma,p)$ be a graph with rank~$n+m$, weighted degree~$d_w$, basepoint~$p$, and~$m$ pairwise disjoint distinguished cycles, disjoint from~$p$. The following hold:
 \begin{enumerate}
  \item If $n>2d_w+m$ then~$\Gamma$ has a loop at the basepoint.
  \item If $n>(3d_w+m)/2$ then~$\Gamma$ has either a loop at the basepoint or a theta graph wedge summand.
  \item If $m>2d_w$ then~$\Gamma$ has a lollipop at the basepoint.
  \item If $m>3d_w/2$ then~$\Gamma$ has a lollipop or a double lollipop at the basepoint.
 \end{enumerate}
\end{lemma}

\begin{proof}
 Since~$p$ is not contained in a distinguished cycle, we have that the degree~$d_0$ is $d_0=d_w+m$. The first two parts of the lemma then follow from \cite[Lemma~5.2]{hatcher98}. Next suppose that there are no lollipops at~$p$, and we want to show that $m\le 2d_w+1$. We will induct on~$n$. If $n=0$ then every undistinguished edge in~$\Gamma$ is a separating edge. Let~$\Gamma'$ be the graph obtained by blowing down every undistinguished edge. Now~$\Gamma'$ is a \emph{cactus graph} as in \cite{collins89}, i.e., every edge is contained in a unique reduced cycle. Note that~$\Gamma'$ is no longer in $\Delta K_0^m$, since the distinguished cycles are not disjoint, but~$\Gamma'$ has the same weighted degree~$d_w$ as~$\Gamma$. Let~$b'$ be the number of cycles in~$\Gamma'$ at~$p$ and $c'=m-b'$ the number of cycles not at~$p$. Since~$\Gamma$ had no lollipops (or loops) at~$p$,~$\Gamma'$ has no loops at~$p$. This tells us that $b'\le c'$, and since $m=b'+c'$ we see that $m\le 2c'$. It is also clear that in~$\Gamma'$, $c'=m-val(p)/2=d_w$, so indeed $m\le 2d_w$. This finishes the base case, and we also note that if additionally~$\Gamma$ has no double lollipops then $b'\le c'/2$, so $m\le 3c'/2=3d_w/2$.

 Now assume $n>0$. Then there exists a undistinguished edge~$\epsilon$ that is not a separating edge. Let~$\Gamma_1$ be the graph obtained from~$\Gamma$ by removing~$\epsilon$, and then if any bivalent vertices~$v\neq p$ arise (or univalent vertices~$v$), blowing down one of the edges containing~$v$. Then~$\Gamma_1$ is a connected graph with undistinguished rank~$n-1$ and~$m$ distinguished cycles. Let $a\in\{0,1,2\}$ be such that the weighted degree $d_w(\Gamma_1)$ of~$\Gamma_1$ is $d_w-a$. In particular $a=0$ if and only if~$\epsilon$ is a loop at~$p$, and $a=1$ if and only if~$p$ is an endpoint of~$\epsilon$ and~$\epsilon$ is not a loop. The graph~$\Gamma_1$ has at most two lollipops at the basepoint, say there are~$b$ of them, so $b\in\{0,1,2\}$. Let~$\Gamma_2$ be the graph obtained by removing all lollipops at~$p$ in~$\Gamma_1$. Then the weighted degree $d_w(\Gamma_2)$ of~$\Gamma_2$ is the same as~$\Gamma_1$, the undistinguished rank is~$n-1$, and there are~$m-b$ distinguished cycles. By induction, $m-b\le 2(d_w-a)$, so $m\le 2d_w-(2a-b)$. It now suffices to show that $2a\ge b$. Clearly if $a=0$ then $b=0$, so suppose $a>0$. Then the only case to check is when $b=2$. But then~$p$ cannot be an endpoint of~$\epsilon$, so $a=2$ and the result follows. We remark that the stronger statement $a\ge b$ even holds.
 
 Lastly suppose that~$\Gamma$ has no lollipops or double lollipops at~$p$. Let $b\in\{0,1,2\}$ be the number of lollipops in~$\Gamma_1$ and $c\in\{0,1,2\}$ the number of double lollipops in~$\Gamma_1$, so $b+c\in\{0,1,2\}$. Let~$\Gamma_3$ be the graph obtained by removing all lollipops and double lollipops at~$p$ in~$\Gamma_1$. Let $a\in\{0,1,2,3,4\}$ be such that~$\Gamma_3$ has weighted degree $d_w-a$. Again, $a=0$ if and only if~$\epsilon$ is a loop at~$p$. Also, if~$\epsilon$ is not a loop but~$p$ is an endpoint of~$\epsilon$ then $a=1+c$, and otherwise $a=2+c$. See Figure~\ref{fig:degree_and_lollipops} for some examples. By the induction hypothesis $m-(b+2c)\le 3(d_w-a)/2$, so $m\le 3d_w/2-(3a/2-(b+2c))$. It now suffices to show that $3a\ge 2b+4c$. If $a=0$ then $b=c=0$, so suppose $a>0$. If~$p$ is an endpoint of~$\epsilon$ then $b+c\le1$ and $a=1+c$, so $2b+4c\le 2+2c=2a<3a$. Now suppose~$p$ is not an endpoint of~$\epsilon$, so $b+c\le 2$ and $c=a-2$. Then $2b+4c\le 4+2c=2a<3a$ and we are done. Again, we find that a stronger statement holds, namely $a\ge b+2c$.
\end{proof}

\begin{figure}[htb]
    \centering
    \begin{tikzpicture}%degree and lollipops
  
  \filldraw[lightgray]
   (0,0) to [out=90, in=10, looseness=1] (-1.3,2.1) to [out=190, in=135, looseness=1] (0,0);
  \draw
   (0,0) to [out=90, in=10, looseness=1] (-1.3,2.1) to [out=190, in=135, looseness=1] (0,0);

  \filldraw
   (0,0) circle (3.5pt)
   (0.7,0.5) circle (1.8pt)
   (-0.5,1.8) circle (1.8pt);
  \draw[line width=2.5pt]
   (1.5,1.1) circle (0.3cm);
  \draw[line width=0.7]
   (0,0) -- (1.3,0.9)   (0.7,0.5) to [out=100, in=-20,looseness=1] (-0.5,1.8);

  \node at (-0.8,1.4) {$\Gamma'$};
  \node at (0.3,1.55) {$\epsilon$};
  \node at (0,-0.5) {$a=2$, $b=1$, $c=0$};

\begin{scope}[xshift=4cm]
  \filldraw[lightgray]
   (0,0) to [out=90, in=10, looseness=1] (-1.3,2.1) to [out=190, in=135, looseness=1] (0,0);
  \draw
   (0,0) to [out=90, in=10, looseness=1] (-1.3,2.1) to [out=190, in=135, looseness=1] (0,0);

  \filldraw
   (0,0) circle (3.5pt)
   (0.7,0.5) circle (1.8pt)
   (-0.5,1.8) circle (1.8pt);
  \draw[line width=2.5pt]
   (1.5,2.4) circle (0.3cm)
   (1.5,1.1) circle (0.3cm);
  \draw[line width=0.7]
   (1.5,1.4) -- (1.5,2.1)
   (0,0) -- (1.3,0.9)   (0.7,0.5) to [out=100, in=-20,looseness=1] (-0.5,1.8);

  \node at (-0.8,1.4) {$\Gamma'$};
  \node at (0.3,1.55) {$\epsilon$};
  \node at (0,-0.5) {$a=3$, $b=0$, $c=1$};
\end{scope}

\begin{scope}[yshift=-4cm]
  \filldraw[lightgray]
   (0,0) to [out=90, in=10, looseness=1] (-1.3,2.1) to [out=190, in=135, looseness=1] (0,0);
  \draw
   (0,0) to [out=90, in=10, looseness=1] (-1.3,2.1) to [out=190, in=135, looseness=1] (0,0);

  \filldraw
   (0,0) circle (3.5pt)
   (1.5,1.8) circle (1.8pt);
  \draw[line width=2.5pt]
   (1.5,2.4) circle (0.3cm)
   (1.5,1.1) circle (0.3cm);
  \draw[line width=0.7]
   (1.5,1.4) -- (1.5,2.1)
   (0,0) -- (1.3,0.9)   (0,0) to [out=80, in=210,looseness=1] (1.5,1.8);

  \node at (-0.8,1.4) {$\Gamma'$};
  \node at (0.6,1.5) {$\epsilon$};
  \node at (0,-0.5) {$a=2$, $b=0$, $c=1$};
\end{scope}

\begin{scope}[yshift=-4cm, xshift=4cm]
  \filldraw[lightgray]
   (0,0) to [out=90, in=10, looseness=1] (-1.3,2.1) to [out=190, in=135, looseness=1] (0,0);
  \draw
   (0,0) to [out=90, in=10, looseness=1] (-1.3,2.1) to [out=190, in=135, looseness=1] (0,0);

  \filldraw
   (0,0) circle (3.5pt)
   (0.7,0.5) circle (1.8pt)
   (0.3,1.9) circle (1.8pt);
  \draw[line width=2.5pt]
   (0.3,2.4) circle (0.3cm)
   (0.3,1.3) circle (0.3cm)
   (1.5,2.4) circle (0.3cm)
   (1.5,1.1) circle (0.3cm);
  \draw[line width=0.7]
   (1.5,1.4) -- (1.5,2.1)
   (0,0) -- (1.3,0.9)
   (0,0) -- (0.3,1)   (0.3,1.6) -- (0.3,2.1)
   (0.7,0.5) to [out=60, in=0, looseness=1.2] (0.3,1.9);

  \node at (-0.8,1.4) {$\Gamma'$};
  \node at (0.95,1.7) {$\epsilon$};
  \node at (0,-0.5) {$a=4$, $b=0$, $c=2$};
\end{scope}

\end{tikzpicture}
    \caption{}
    \label{fig:degree_and_lollipops}
\end{figure}

\begin{remark}\label{better_bounds}
 In the last two paragraphs of the proof, it is interesting that the induction would have run even with sharper bounds. In fact, whatever the best possible bound is for the $n=0$ case automatically extends to all cases, as long as the slope is not less than~$1$. In particular, we can detect ``triple lollipops,'' ``quadruple lollipops,'' as so forth, with increasingly better bounds. Ultimately, we find that whenever $m>d_w$, there is always some non-trivial wedge summand that is an iterated wedge of lollipops. However, since we currently do not have a way to make use of this fact to get better bounds for homological stability, we will content ourselves with just detecting lollipops and double lollipops.
\end{remark}

\begin{proposition}[Stability in $n$]\label{orbit_spaces_n_stable}
 The map
$$\nu \colon \nabla K_{n,k}^m/G_n^m \hookrightarrow \nabla K_{n+1,k}^m/G_{n+1}^m$$
is a homeomorphism for $2k+m<n+1$ and a homotopy equivalence for $(3k+m)/2<n+1$.
\end{proposition}

\begin{proof}
 The proof is very similar to the proof of Proposition~5.4 in \cite{hatcher98}. If $2k+m<n+1$ then every~$\Gamma$ in $\nabla K_{n+1,k}^m/G_{n+1}^m$ has a loop at~$p$, so~$\nu$ is a homeomorphism. Now suppose $(3k+m)/2<n+1$, and let~$\Gamma$ be a vertex not in the image of~$\nu$. Then~$\Gamma$ has no loops at~$p$ but does have at least one theta graph wedge summand. Let~$\Theta$ be the subgraph of~$\Gamma$ consisting of all such theta graphs at~$p$, say there are~$r\ge 1$ of them. Then $\Gamma=\Theta\vee\Gamma'$, for some~$\Gamma'$ with rank $n+m+1-2r$. Now, the open star of~$\Gamma$ in $\nabla K_{n+1,k}^m/G_{n+1}^m$ is the product of open stars of~$\Theta$ in $\nabla K_{2r,r}^0/G_{2r}^0$ and~$\Gamma'$ in $\nabla K_{n+1-2r,k-r}^m/G_{n+1-2r}^m$. The former consists of a single simplex, since all non-loop edges in~$\Theta$ are equivalent under automorphisms of~$\Theta$; moreover, every other vertex of this star has lower weighted degree since blowing down any edge reduces~$d_w$ by~$1$. So, collapsing any non-loop edge of~$\Theta$ gives a deformation retraction of the star of~$\Gamma$ into the image of~$\nu$.
\end{proof}

As a remark, in \cite{hatcher98b} some bounds are given to detect wedge summands of higher degree, and the possibility of collapsing these in a similar way to the theta wedge summands is examined. In the present situation though, this collapse could cause~$p$ to become distinguished, which is a problem. Hence we cannot immediately improve the bound to $(5k+m)/4<n+1$, as was done for the~$m=0$ case in \cite{hatcher98b}. It seems likely that we could nonetheless improve this bound by directly inspecting examples with low (weighted) degree, in the spirit of \cite{hatcher98b}, but we leave this for future work.

\begin{proposition}[Stability in $m$]\label{orbit_spaces_m_stable}
 Let $\nabla Q_{n,k}^m \defeq \nabla K_{n,k}^m/\SAut_n^m$. The map
$$\mu \colon \nabla Q_{n,k}^m \hookrightarrow \nabla Q_{n,k}^{m+1}$$
is a homeomorphism for $2k<m+1$ and a homotopy equivalence for $3k/2<m+1$.
\end{proposition}

\begin{proof}
 If $2k<m+1$ then every~$\Gamma$ in $\nabla Q_{n,k}^{m+1}$ has a lollipop at~$p$, so~$\mu$ is a homeomorphism. Now suppose $3k/2<m+1$, and let~$\Gamma$ be a vertex not in the image of~$\mu$. Then~$\Gamma$ has no lollipops at~$p$ but does have at least one double lollipop. Let $\Lambda\Lambda$ be the subgraph of~$\Gamma$ consisting of all double lollipops at~$p$, say there are~$r\ge 1$ of them. Then $\Gamma=\Lambda\Lambda\vee\Gamma'$, for some~$\Gamma'$ with rank $n+m+1-2r$. The open star of~$\Gamma$ in $\nabla Q_{n,k}^{m+1}$ is the product of open stars of $\Lambda\Lambda$ in $\nabla Q_{0,r}^{2r}$ and~$\Gamma'$ in $\nabla Q_{n-2r,k-r}^{m+1}$. We claim that there is a retraction of the former that yields a retraction of the star of~$\Gamma$ into the image of~$\mu$, similar to the previous proof. Consider the height function~$h$ from Section~\ref{sec:morse}, thought of on~$\nabla K_{0,r}^{2r}$, and note that since~$h$ only depends on~$\rho$ inasmuch as~$\rho$ determines which cycles are distinguished,~$h$ descends to a function~$\overline{h}$ on $\nabla Q_{0,r}^{2r}$. Since $\nabla Q_{0,r}^{2r}$ is not simplicial we think of~$\overline{h}$ as a height function in the sense of~\cite{bux99}. It now suffices to show that the descending link $\overline{\dlk}(\Gamma)$ is contractible.

 There are three homeomorphism types of double lollipops, depending on where the first lollipop is wedged to the second. If it is wedged to a point in the interior of the stick, call this Type~$1$. If it is wedged to a point on the distinguished cycle not in the stick, call this Type~$2$. If it is wedged to the vertex shared by the loop and the stick call this Type~$3$. See Figure~\ref{fig:double_lollipops}. If $\Lambda\Lambda$ has a double lollipop of Type~$1$ then blowing down the edge connecting the wedge point to~$p$ is descending (with essential term~$d_0$). Moreover, every simplex in $\overline{\dlk}(\Gamma)$ is compatible with this move since descending blow-ups cannot affect double lollipops of Type~$1$, so it is a cone point of $\overline{\dlk}(\Gamma)$. Next, if $\Lambda\Lambda$ has a double lollipop of Type~$2$, then blowing down either edge connecting the wedge point to the top of the stick is descending (with essential term~$d_0$). These edges differ by a homeomorphism of~$\Gamma$, so they actually correspond to the same blow-down. Again, every simplex in $\overline{\dlk}(\Gamma)$ is compatible with this move since descending blow-ups cannot affect double lollipops of Type~$2$, so it is a cone point of~$\overline{\dlk}(\Gamma)$. Finally suppose $\Lambda\Lambda$ has a double lollipop of Type~$3$. Consider the blow-up that pushes the base of the first cycle away from the wedge point, creating a double lollipop of Type~$1$. This is descending, with essential term~$m_1$, and since descending (admissible) blow-downs cannot affect double lollipops of Type~$3$, it is a cone point for $\overline{\dlk}(\Gamma)$. We conclude that attaching~$\Gamma$ does not change the homotopy type, by \cite[Lemma~4]{bux99}, so the result follows.
\end{proof}

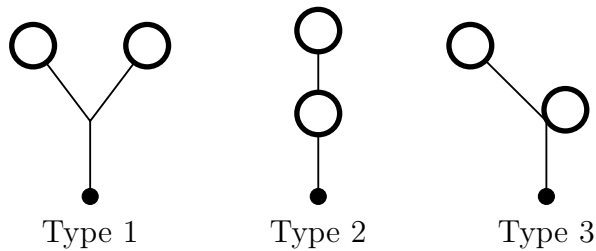
\begin{figure}[htb]
    \centering
    \begin{tikzpicture}%double lollipops
  \draw[line width=0.7pt]
   (0,0) -- (0,1)
   (0,1) -- (-0.6,1.8)
   (0,1) -- (0.6,1.8);
  \draw[line width=2pt]
   (-0.75,2) circle (8pt)
   (0.75,2) circle (8pt);
  \filldraw
   (0,0) circle (3pt);
  \node at (0,-0.5) {Type $1$};

  \begin{scope}[xshift=3cm]
   \draw[line width=0.7pt]
   (0,0) -- (0,0.8)
   (0,1.4) -- (0,1.9);
  \draw[line width=2pt]
   (0,1.1) circle (8pt)
   (0,2.2) circle (8pt);
  \filldraw
   (0,0) circle (3pt);
  \node at (0,-0.5) {Type $2$};
  \end{scope}

  \begin{scope}[xshift=6cm]
   \draw[line width=0.7pt]
   (0,0) -- (0,1)
   (0,1) -- (-0.8,1.8);
  \draw[line width=2pt]
   (0.25,1.15) circle (8pt)
   (-1,2) circle (8pt);
  \filldraw
   (0,0) circle (3pt);
  \node at (0,-0.5) {Type $3$};
  \end{scope}

\end{tikzpicture}
    \caption{Types of double lollipops.}
    \label{fig:double_lollipops}
\end{figure}

There is evidence to suggest that the descending links $\overline{\dlk}(\Gamma)$ are always contractible whenever there is a non-trivial wedge summand that is an iterated wedge of lollipops. As indicated by Remark~\ref{better_bounds}, this would imply that~$\mu$ is a homotopy equivalence whenever~$k\le m$. From this we would also recover the fact that $\SAut_0^m$ has trivial rational homology. For now though, we will content ourselves with the double lollipop situation.

Since~$\nu$ is natural with respect to $G_n^m \hookrightarrow G_{n+1}^m$ and~$\mu$ is natural with respect to $\SAut_n^m \hookrightarrow  \SAut_n^{m+1}$, we can now prove our main results.

\begin{proof}[Proof of Theorem~\ref{hom_stab_thm}]
 We know that for $0\le k< N$, if $(3k+m)/2<n+1$ then
 $$H_i(G_n^m;\Q)\to H_i(G_{n+1}^m;\Q)$$
 is an isomorphism for all $i<k$, by Lemma~\ref{gp_to_space_hlgy} and Proposition~\ref{orbit_spaces_n_stable}. Assume $n\ge (3(i+1)+m)/2$, so in particular $n\ge 2$, and set $k=i+1$. Then $(3k+m)/2<n+1$ and $k\le (2n-m)/3$, which is less than~$N$ since $n\ge 2$. The result now follows.
\end{proof}

Note that when~$m=0$, so $G_n^0=\Aut(F_n)$, we recover the stability bound for $\Aut(F_n)$ given in \cite{hatcher98}, though not the improved one given in \cite{hatcher98b}.

\begin{proof}[Proof of Theorem~\ref{hom_stab_thm2}]
 We know that for $0\le k< N$, if $3k/2<m+1$ then
 $$H_i(\SAut_n^m;\Q) \to H_i(\SAut_n^{m+1};\Q)$$
 is an isomorphism for all $i<k$, by Lemma~\ref{gp_to_space_hlgy} and Proposition~\ref{orbit_spaces_m_stable}. If $n=0$ then the homology groups are all~$0$ by \cite{griffin12,wilson12}, so we can assume $n\ge 1$. Suppose $m>(3i+1)/2$, so in particular $m\ge 1$, and set $k=i+1$. Then $3k/2=3(i+1)/2<m+1$, and also since $n,m\ge 1$ we get $k<(2m+2)/3\le 2n+m-1=N$, so $k< N$. The result now follows.
\end{proof}

\section{Connectivity}\label{sec:connectivity}

The rest of this paper is devoted to proving the Generalized Degree Theorem, Theorem~\ref{thrm:sublev_conn}. This amounts to analyzing the connectivity of descending links of vertices in~$\Delta K_n^m$, with the main result being Corollary~\ref{desc_lk_conn}. In reading these subsections, the reader may find it helpful to refer to the corresponding sections in \cite{zar_degree_thm}, which cover what amounts here to the classical~$m=0$ case.

We first collect some natural definitions that will be used in these subsections, including the important notion of a \emph{decisive edge} in a graph.

\begin{definition}[Edges in graphs]\label{def:more_graph_defs}
 For an edge~$\epsilon$ in a basepointed graph~$\Gamma$ with vertices~$v$ and~$v'$, we call~$\epsilon$ \emph{horizontal} if~$d(p,v)=d(p,v')$. Otherwise we call~$\epsilon$ \emph{vertical}. If $\epsilon$ is vertical, by comparing $d(v,p)$ and $d(v',p)$ we get a natural notion of the \emph{top} vertex and \emph{bottom} vertex of~$\epsilon$. A half-edge may also have either a top or a bottom. If a vertex~$v$ has only one incident vertical edge~$\epsilon$ with~$v$ as its top, we call that edge \emph{decisive at~$v$}. In other words, if every minimal-length path from~$v$ to~$p$ begins with~$\epsilon$, then~$\epsilon$ is decisive at~$v$. If an edge~$\epsilon$ in~$\Gamma$ is decisive at its top vertex we call it a decisive edge. For example any separating edge is decisive.
\end{definition}

\subsection{Connectivity of the descending down-link}\label{sec:down_conn}

In this section we analyze the down-link of~$\Gamma$. In order to get an induction to run, we will need to lift the restriction on the valency of vertices. Our height function $h$ does not work well with such graphs though, for instance the trivalency of non-basepoint vertices is crucial to the fact that blowing down a forest~$F$ either increases $n_{D(F)}$ or decreases $d_{D(F)}$. Thanks to Lemma~\ref{which_forests_desc} though, we have a condition on forests that is equivalent to being descending for graphs $\Gamma\in\Delta K_n^m$, and does not refer to the functions $n_i$ or $d_i$. For lack of a more clever name, we will call such forests \emph{good} (defined below). For the rest of this subsection,~$\Gamma$ is a connected graph with basepoint~$p$ and~$m$ disjoint distinguished cycles, with no restriction on the valency of vertices. The definitions of~$\Lambda_i$ and~$m_i$ remain valid, and are as given above. A reduced, non-self-intersecting edge path~$\gamma$ in~$\Gamma$ will be called an \emph{arc} if both of its endpoints lie in~$\Lambda_{D(\gamma)}$.

\begin{definition}[Good forests]\label{def:good_forests}
 Let~$F$ be an admissible forest in~$\Gamma$. Define
 $$\Delta m_i(\Gamma,F)\defeq m_i(\Gamma/F)-m_i(\Gamma)$$
 for any~$i$. Now let $i\defeq D(F)$. If $\Delta m_i(\Gamma,F)<0$ call~$F$ \emph{base-decreasing}, if $\Delta m_i(\Gamma,F)>0$ call~$F$ \emph{base-increasing} and if $\Delta m_i(\Gamma,F)=0$ call~$F$ \emph{base-preserving}. If~$F$ connects vertices in $\Lambda_i$, or equivalently if~$F$ contains an arc~$\gamma$ with $D(\gamma)=D(F)$, call~$F$ \emph{arced}. If~$F$ does not connect vertices in $\Lambda_i$, call~$F$ \emph{arc-free}. Finally, if~$F$ is base-decreasing, or if it is base-preserving and arc-free, call~$F$ \emph{good}. A forest is \emph{bad} if it is not good.
\end{definition}

Lemma~\ref{which_forests_desc} says that for any $\Gamma\in\Delta K_n^m$, a forest~$F$ in~$\Gamma$ is descending if and only if it is good.

\begin{remark}[Good/bad edges and distinguished paths]\label{edges_and_paths}
 There are a few important technical observations about single edge forests that we collect here. A vertical edge is arc-free and cannot be base-decreasing, and a distinguished vertical edge cannot be base-increasing, so must be base-preserving and arc-free, hence good. A horizontal edge is arced and cannot be base-increasing, and a base-decreasing horizontal edge must be distinguished. Hence a horizontal edge is good if and only if it is distinguished and base-decreasing, i.e., connects two base vertices.

 It is also easy to see whether an edge path~$\gamma$ in a distinguished cycle $C$ is good or bad. Such a~$\gamma$ cannot be base-increasing, so if~$\gamma$ is arc-free then it is automatically good. If~$\gamma$ is arced and $D(\gamma)=i_C$, then~$\gamma$ contains an arc connecting base vertices and so is base-decreasing, hence good. If~$\gamma$ is arced and $D(\gamma)>i_C$ then it is base-preserving, hence bad. To summarize,~$\gamma$ is bad if it is arced and $D(\gamma)>i_C$, and otherwise is good. See Figure~\ref{fig:dstg_edge_path} for some examples.
\end{remark}

\begin{figure}[htb]
    \centering
    \begin{tikzpicture}[scale=0.8]%good and bad distinguished edge paths
  
  \draw[line width=4pt]
   (0,0) -- (-1.5,0) to [out=80, in=180,looseness=1] (0,2) to [out=0, in=100,looseness=1] (1.5,0) -- (0,0);

  \draw[line width=1.3pt, gray]
   (-1.5,0) to [out=80, in=180,looseness=1] (0,2) to [out=0, in=100,looseness=1] (1.5,0);

  \node at (0,-0.5) {good};

 \begin{scope}[xshift=4cm]
  \draw[line width=4pt]
   (0,0) -- (-1.5,0) to [out=80, in=180,looseness=1] (0,2) to [out=0, in=100,looseness=1] (1.5,0) -- (0,0);

  \draw[line width=1.3pt, gray]
   (-1.5,0) to [out=80, in=180,looseness=1] (0,2) to [out=0, in=110,looseness=1] (1.32,0.8);

  \node at (0,-0.5) {good};
 \end{scope}

 \begin{scope}[xshift=8cm]
  \draw[line width=4pt]
   (0,0) -- (-1.5,0) to [out=80, in=180,looseness=1] (0,2) to [out=0, in=100,looseness=1] (1.5,0) -- (0,0);

  \draw[line width=1.3pt, gray]
   (-1.32,0.8) to [out=70, in=180,looseness=1] (0,2) to [out=0, in=110,looseness=1] (1.32,0.8);

  \node at (0,-0.5) {bad};
 \end{scope}

 \begin{scope}[xshift=12cm]
  \draw[line width=4pt]
   (0,0) -- (-1.5,0) to [out=80, in=180,looseness=1] (0,2) to [out=0, in=100,looseness=1] (1.5,0) -- (0,0);

  \draw[line width=1.3pt, gray]
   (-1.32,0.8) to [out=70, in=180,looseness=1] (0,2) to [out=0, in=140,looseness=1] (0.9,1.6);

  \node at (0,-0.5) {good};
 \end{scope}

\end{tikzpicture}
    \caption{Good and bad distinguished edge paths.}
    \label{fig:dstg_edge_path}
\end{figure}
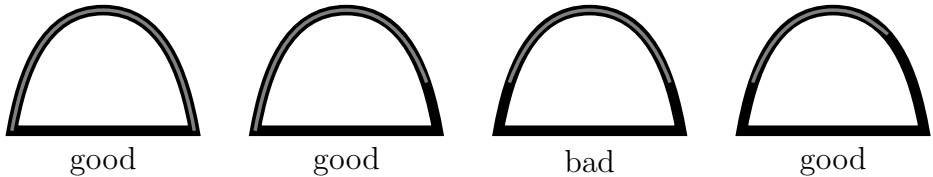

\noindent\textbf{Posets of forests:} Let $P(\Gamma)$ be the poset of good admissible forests in~$\Gamma$. For $\Gamma\in\Delta K_n^m$, the down-link of~$\Gamma$ is the geometric realization $|P(\Gamma)|$ of $P(\Gamma)$, so the goal of this section is to calculate the homotopy type of $|P(\Gamma)|$. For the rest of this section we will omit the vertical bars, and just refer to $P(\Gamma)$ as having a homotopy type. For each edge~$\epsilon$ of~$\Gamma$, let $P_1(\Gamma,\epsilon)$ be the poset of all good admissible forests except the forest just consisting of~$\epsilon$, and let $P_0(\Gamma,\epsilon)\subseteq P_1(\Gamma,\epsilon)$ be the poset of good admissible forests that do not contain~$\epsilon$. Whenever~$\Gamma$ and~$\epsilon$ are understood from context we will just write~$P$,~$P_1$ and~$P_0$. We call~$P_1(\Gamma,\epsilon)$ the \emph{deletion} of~$\epsilon$, and $P_0(\Gamma,\epsilon)$ the \emph{strong deletion} of~$\epsilon$.

\begin{lemma}[Strong deletion of distinguished edge]\label{P0_dstgd_edge}
 For an admissible distinguished edge~$\epsilon$, $P_0(\Gamma,\epsilon)$ is contractible.
\end{lemma}

\begin{proof}
 Let $C$ be the distinguished cycle containing~$\epsilon$, and let~$\phi$ be the forest consisting of all edges of~$C$ other than~$\epsilon$. Since $D(\phi)=i_C$,~$\phi$ is good by Remark~\ref{edges_and_paths}. Let $f \colon P_0\to P_0$ be given by
 $$F\mapsto F\cup\phi.$$
 We claim that for~$F\in P_0$, $F\cup\phi$ is an admissible good forest, so~$f$ is well defined. Since~$\epsilon\not\in F$, and~$F$ is admissible, it is clear that~$F\cup\phi$ is an admissible forest. Let~$\phi'$ be the image of~$\phi$ in~$\Gamma/F$, so
 $$\Gamma/F\cup\phi=(\Gamma/F)/\phi'.$$
 By Remark~\ref{edges_and_paths},~$\phi'$ is not base-increasing, which tells us that if~$F$ is base-decreasing then so is~$F\cup\phi$, and so the claim follows. The other way~$F$ can be good is if it is base-preserving and arc-free. Then by the same argument,~$F\cup\phi$ is not base-increasing, so it suffices to show that if~$F\cup\phi$ is arced, then it is base-decreasing. If~$\phi$ itself is arced then it must be base-decreasing, which implies~$F\cup\phi$ is base-decreasing. Suppose instead that~$\phi$ is arc-free (and recall that~$F$ is too). For $F\cup\phi$ to be arced then, we need that $D(F)=D(\phi)=:i$ and that~$\phi'$ is arced. But by Remark~\ref{edges_and_paths}, if $\phi'$ is arced then it is base-decreasing, in which case $F\cup\phi$ is base-decreasing, so the claim follows in this case as well.

 We conclude that~$f$ is well defined, and so it follows from \cite[Section~1.5]{quillen78} that~$P_0$ is contractible.
\end{proof}

\noindent\textbf{Optimal edges:} For an admissible edge~$\epsilon$ with endpoints $v_1$ and $v_2$, call~$\epsilon$ \emph{maximally distant} if among all admissible edges,~$\epsilon$ maximizes the quantity $d(p,v_1)+d(p,v_2)$. This quantity is even (resp. odd) if~$\epsilon$ is horizontal (resp. vertical). Hence all maximally distant edges have the same orientation, i.e., horizontal or vertical. If a maximally distant edge~$\epsilon$ maximizes the quantity $\Delta m_{D(\epsilon)}(\Gamma,\epsilon)$ among all maximally distant edges, call~$\epsilon$ \emph{optimal}. Note that if there exists a \emph{good} optimal edge, then either every maximally distant edge is vertical and good, or else every maximally distant edge is horizontal and connects base vertices (and so is good).

\begin{proposition}[From $P_0$ to $P_1$]\label{P1_P0}
 Let~$\epsilon$ be an optimal maximally distant edge. Then $P_1(\Gamma,\epsilon)$ is homotopy equivalent to $P_0(\Gamma,\epsilon)$.
\end{proposition}

\begin{proof}
 We begin by finding an intermediate poset that is easily seen to be homotopy equivalent to $P_0$. Let $P_{\frac{1}{2}}=P_{\frac{1}{2}}(\Gamma,\epsilon)$ be the subcomplex of~$P$ spanned by good admissible forests~$F$ for which $F \setminus \{\epsilon\}$ is again a (non-empty) good admissible forest. Call $P_{\frac{1}{2}}$ the \emph{sufficiently strong deletion} of $\epsilon$. Clearly
 $$P_0\subseteq P_{\frac{1}{2}}\subseteq P_1.$$
 Let $f \colon P_{\frac{1}{2}} \to P_{\frac{1}{2}}$ be given by $F\mapsto F\setminus\{\epsilon\}$. This is a well defined poset map that is the identity on its image~$P_0$, and so induces a homotopy equivalence between~$P_{\frac{1}{2}}$ and~$P_0$ by \cite[Section~1.3]{quillen78}.

 \noindent\textbf{Case 1: Undistinguished optimal edge:} First suppose that~$\epsilon$ is undistinguished, and we claim that $P_{\frac{1}{2}}=P_1$. Let~$F\in P_1$ and let~$i\defeq D(F)$. We want to show that $F\setminus\{\epsilon\}$ is good. We may assume~$\epsilon$ is (properly) contained in~$F$, which since~$\epsilon$ is maximally distant tells us that $D(F\setminus\{\epsilon\})=i$. If~$\epsilon'$ is the image of~$\epsilon$ in $\Gamma/(F\setminus\{\epsilon\})$ then~$\epsilon'$ is undistinguished, and so cannot be base-decreasing. Hence
 $$\Delta m_i(\Gamma,F)=\Delta m_i(\Gamma/(F\setminus\{\epsilon\}),\epsilon')+\Delta m_i(\Gamma,F\setminus\{\epsilon\})\ge\Delta m_i(\Gamma,F\setminus\{\epsilon\}).$$
 Clearly if~$F$ is arc-free then $F\setminus\{\epsilon\}$ is too. From this fact and the above equation, we conclude that if~$F$ is good then so is $F\setminus\{\epsilon\}$. We remark that so far we have not used the hypothesis that~$\epsilon$ is optimal.

 \noindent\textbf{Case 2: Distinguished optimal edge:} Now assume~$\epsilon$ is distinguished, so we know $\Delta m_{D(\epsilon)}(F,\epsilon)\le0$. We have to do a bit more work in this case. Define a height function~$e$ on~$P_1$ as follows. For~$F\in P_1$, if $F\in P_{\frac{1}{2}}$ set~$e(F)=0$ and otherwise let~$e(F)$ be the number of edges in~$F$. Since adjacent vertices (forests) in $P_1\setminus P_{\frac{1}{2}}$ have different~$e$ values, we can build up from~$P_{\frac{1}{2}}$ to~$P_1$ by gluing in vertices along their descending links. We claim these descending links are contractible, so by \cite[Corollary~2.6]{bestvina97} the homotopy type does not change, and the result follows. The descending link of~$F$ is the join of two subcomplexes, which we will call the \emph{out-link} and the \emph{in-link}. The out-link is spanned by forests in~$P_{\frac{1}{2}}$ containing~$F$, and the in-link by forests in~$P_1$ properly contained in~$F$. It suffices to show that the in-link is contractible.

 \noindent\textbf{Calculating $\Delta m_i$:} A forest~$F$ in~$P_1$ but not in~$P_{\frac{1}{2}}$ is characterized by~$F$ being good and $F\setminus\{\epsilon\}$ being bad. This is a relatively specific situation, and we will be able to restrict the possibilities quite a bit. First of all, $\epsilon\subseteq F$, and~$\epsilon$ is maximally distant so $D(F\setminus\{\epsilon\})=i\defeq D(F)$. Consider again the equation
 $$\Delta m_i(\Gamma,F)=\Delta m_i(\Gamma/(F\setminus\{\epsilon\}),\epsilon')+\Delta m_i(\Gamma,F\setminus\{\epsilon\}),$$
 where~$\epsilon'$ is the image of~$\epsilon$ in $\Gamma/(F\setminus\{\epsilon\})$. Since~$F$ is good and $F\setminus\{\epsilon\}$ is bad, and since if~$F$ is arc-free then so is $F\setminus\{\epsilon\}$, it is clear that $\Delta m_i(\Gamma/(F\setminus\{\epsilon\}),\epsilon')$ cannot be~$0$ or~$1$. The only other option is that it equals $-1$. This implies that~$\epsilon'$ connects base vertices, and so in particular~$F$ must be arced, with an arc containing~$\epsilon$ and connecting base vertices. Since~$F$ is good it therefore must be base-decreasing, and so we conclude that
 \begin{align*}
  &\Delta m_i(\Gamma,F)=-1\text{,} \\
  &\Delta m_i(\Gamma/(F\setminus\{\epsilon\}),\epsilon')=-1 \\
  \text{and }&\Delta m_i(\Gamma,F\setminus\{\epsilon\})=0\text{.}
 \end{align*}
 Then since $F\setminus\{\epsilon\}$ is bad, it must be arced.

 \noindent\textbf{A crucial arc in $F$:} Let $C$ be the distinguished cycle containing~$\epsilon$. Since $\epsilon\subseteq F$ and~$F$ is admissible, we know $F\cap C$ is a forest. Let~$\gamma'$ be the connected edge path in $F\cap C$ containing~$\epsilon$. By the previous paragraph, we see that~$\gamma'$ must contain an arc at level $D(F)$ that in turn contains~$\epsilon$. Let~$\gamma$ be the shortest arc in~$\gamma'$ containing~$\epsilon$ with $D(\gamma)=D(F)$. If $\gamma=\epsilon$ then $D(F)=D(\epsilon)$, and~$\epsilon$ being both an arc and an optimal edge implies that it, and so every edge of~$F$, is horizontal and connects base vertices. Hence $F\setminus\{\epsilon\}$ is base-decreasing, which we know is not the case. We can therefore assume~$\gamma$ properly contains~$\epsilon$. According to Remark~\ref{edges_and_paths},~$\gamma$ is base-decreasing, hence good, and it is easy to see that $\gamma\setminus\{\epsilon\}$ is arc-free and non-base-increasing, so also good. Since $F\setminus\{\epsilon\}$ is bad, this means~$\gamma$ does not equal~$F$, so~$\gamma$ is really in the in-link. See Figure~\ref{fig:tricky_detail} for an idea of~$\gamma'$ and~$\gamma$.

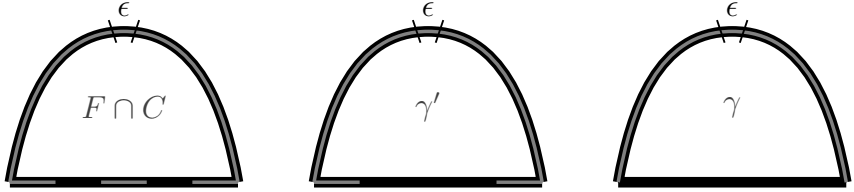
\begin{figure}[htb]
    \centering
    \begin{tikzpicture}%tricky detail
  
  \draw[line width=4pt]
   (-1.5,0) -- (1.5,0)
   (-1.5,0) to [out=80, in=180,looseness=1] (0,2) to [out=0, in=100,looseness=1] (1.5,0);

  \draw[line width=1.3pt, gray]
   (-1.5,0) -- (-0.9,0)   (-0.3,0) -- (0.3,0)   (0.9,0) -- (1.5,0)
   (-1.5,0) to [out=80, in=180,looseness=1] (0,2) to [out=0, in=100,looseness=1] (1.5,0);

  \draw[line width=0.7]
   (-0.2,2.15) -- (-0.1,1.85)   (0.2,2.15) -- (0.1,1.85);

  \node at (0,2.3) {$\epsilon$};

  \node[darkgray] at (0,1) {$F\cap C$};

  \begin{scope}[xshift=4cm]
   \draw[line width=4pt]
   (-1.5,0) -- (1.5,0)
   (-1.5,0) to [out=80, in=180,looseness=1] (0,2) to [out=0, in=100,looseness=1] (1.5,0);

  \draw[line width=1.3pt, gray]
   (-1.5,0) -- (-0.9,0)   (0.9,0) -- (1.5,0)
   (-1.5,0) to [out=80, in=180,looseness=1] (0,2) to [out=0, in=100,looseness=1] (1.5,0);

  \draw[line width=0.7]
   (-0.2,2.15) -- (-0.1,1.85)   (0.2,2.15) -- (0.1,1.85);

  \node at (0,2.3) {$\epsilon$};

  \node[darkgray] at (0,1) {$\gamma'$};
  \end{scope}

  \begin{scope}[xshift=8cm]
   \draw[line width=4pt]
   (-1.5,0) -- (1.5,0)
   (-1.5,0) to [out=80, in=180,looseness=1] (0,2) to [out=0, in=100,looseness=1] (1.5,0);

  \draw[line width=1.3pt, gray]
   (-1.5,0) to [out=80, in=180,looseness=1] (0,2) to [out=0, in=100,looseness=1] (1.5,0);

  \draw[line width=0.7]
   (-0.2,2.15) -- (-0.1,1.85)   (0.2,2.15) -- (0.1,1.85);

  \node at (0,2.3) {$\epsilon$};

  \node[darkgray] at (0,1) {$\gamma$};
  \end{scope}

\end{tikzpicture}
    \caption{$F\cap C$,~$\gamma'$ and~$\gamma$.}
    \label{fig:tricky_detail}
\end{figure}

 \noindent\textbf{Contractibility of the in-link:} The idea now is to retract the in-link to the relative star of~$\gamma$. We claim that for any~$F'$ in the in-link, $F'\cup\gamma$ is also in the in-link. It is clear that $F'\cup\gamma$ is admissible, since it is contained in~$F$. If $\gamma\subseteq F'$ there is nothing to show, so we can assume rather that the image of~$\gamma$ in $\Gamma/F'$ is an arc, which necessarily connects base vertices and so is base-decreasing. Since~$F'$ is good we conclude that $F'\cup\gamma$ is base-decreasing, and so is also good.

 It remains only to show that $F'\cup\gamma\neq F$. We claim that for any $\emptyset\neq\delta\subseteq\gamma$, $F\setminus\delta$ is bad. This can phrased colloquially as: if removing~$\epsilon$ from~$F$ turns it bad, then removing any part of~$\gamma$ from~$F$ turns it bad. Since~$F'$ is good, this will then imply that $F'\cup\gamma\neq F$. Note that if $\epsilon\not\in\delta$ and $F\setminus\delta$ is good, the connected component of $(F\setminus\delta)\cap C$ containing~$\epsilon$ does not connect base vertices, so by the previous paragraphs $F\setminus\delta\in P_{\frac{1}{2}}$, i.e., $(F\setminus\delta)\setminus\{\epsilon\}$ is good. In particular if $F\setminus(\delta\cup\{\epsilon\})$ is bad then so is $F\setminus\delta$, so we can assume without loss of generality that $\epsilon\subseteq\delta$. Since $F\setminus\{\epsilon\}$ is arced we have $D(F\setminus\gamma)=i$, and so $D(F\setminus\delta)=i$. It is clear that
 $$\Delta m_i(\Gamma,F\setminus\delta)\ge\Delta m_i(\Gamma,F\setminus\{\epsilon\})=0,$$
 so to show $F\setminus\delta$ is bad, it suffices to assume it is arc-free and prove it is base-increasing. For $F\setminus\{\epsilon\}$ to be arced and $F\setminus\delta$ to be arc-free, there must exist an arc in $F\setminus\{\epsilon\}$ containing an edge of $\delta\setminus\{\epsilon\}$. In particular, the image of $\delta\setminus\{\epsilon\}$ in $\Gamma/(F\setminus\delta)$ is an arced forest consisting of distinguished edges, with an arc connecting base vertices. This must be base-decreasing, which tells us that
 $$\Delta m_i(\Gamma,F\setminus\delta)>\Delta m_i(\Gamma,F\setminus\{\epsilon\}),$$
 and we are done. The claim now follows, and so $F'\cup\gamma$ is in the in-link. In particular the in-link is contractible by \cite[Section~1.5]{quillen78}.
\end{proof}

\noindent\textbf{Decomposing $P$ using $\epsilon$:} In general if~$\epsilon$ is any admissible good edge, then we have
\begin{align*}
 P(\Gamma)=&P_1(\Gamma,\epsilon)\cup\st(\epsilon)\\
 \text{and }&P_1(\Gamma,\epsilon)\cap\st(\epsilon)=\lk(\epsilon)\text{,}
\end{align*}
where star and link here are taken in $P(\Gamma)$. The previous results provide tools to analyze $P_1(\Gamma,\epsilon)$, and the next lemma tells us something about $\lk(\epsilon)$.

\begin{lemma}[Links in the down-link]\label{forest_mod_edge}
 Let~$\epsilon$ be an optimal edge in~$\Gamma$ such that $\epsilon\in P(\Gamma)$, i.e.,~$\epsilon$ is good. Let~$F$ be an admissible forest properly containing~$\epsilon$. Then $F\in P(\Gamma)$ if and only if $F/\epsilon\in P(\Gamma/\epsilon)$. Moreover, $\lk(\epsilon)\cong P(\Gamma/\epsilon)$.
\end{lemma}

\begin{proof}
 Let $i\defeq D(F)=D(F/\epsilon)$. Since~$\epsilon$ is good, $\Delta m_i(\Gamma,\epsilon)\in\{-1,0\}$. First suppose that $\Delta m_i(\Gamma,\epsilon)=0$, for example if $D(\epsilon)>i$. It is clear that~$F$ is arced if and only if~$F/\epsilon$ is arced. Also,
 $$\Delta m_i(\Gamma,F)=\Delta m_i(\Gamma/\epsilon,F/\epsilon)+\Delta m_i(\Gamma,\epsilon),$$
 so $\Delta m_i(\Gamma,F)=\Delta m_i(\Gamma/\epsilon,F/\epsilon)$. Hence,~$F$ is base-decreasing if and only if~$F/\epsilon$ is, and~$F$ is base-preserving and arc-free if and only if~$F/\epsilon$ is, which implies that~$F\in P(\Gamma)$ if and only if $F/\epsilon\in P(\Gamma/\epsilon)$.

 Next suppose $\Delta m_i(\Gamma,\epsilon)=-1$, so~$D(\epsilon)=i$. We claim that in fact~$F$ and~$F/\epsilon$ must both be base-decreasing, and hence good. We know that~$\epsilon$, and indeed every maximally distant edge, is horizontal and connects base vertices. In particular since $D(\epsilon)=i$, every edge of~$F$ must be maximally distant, and so connects base vertices. Since~$F$ has more than one edge, it is clear that $\Delta m_i(\Gamma,F)\le-2$, so~$F$ is base-decreasing. Also,
$$\Delta m_i(\Gamma/\epsilon,F/\epsilon)=\Delta m_i(\Gamma,F)-\Delta m_i(\Gamma,\epsilon)\le-2+1=-1$$
 so~$F/\epsilon$ is base-decreasing.

 Now consider the map
 $$f \colon \lk(\epsilon) \to P(\Gamma/\epsilon)$$
 sending~$F$ to $F/\epsilon$. This is well-defined by the previous paragraphs, and is clearly injective. We claim that~$f$ is bijective. Let~$\Phi\in P(\Gamma/\epsilon)$. There are two forests in~$\Gamma$ that map to~$\Phi$ under blowing down~$\epsilon$, one that contains~$\epsilon$ and one that does not. Let~$\Phi'$ be the one that does, so $\Phi'\in\lk(\epsilon)$ and $f(\Phi')=\Phi$. If~$\Phi$ was admissible then~$\Phi'$ is too. Also, if~$\Phi$ was good then so is~$\Phi'$, again by the previous paragraphs. So~$f$ is an isomorphism.
\end{proof}

Let~$V$ be the number of vertices in~$\Gamma$ and~$E_{ad}$ the number of admissible edges. The next two results are generalizations of Proposition~3.2 and Lemma~3.3 from \cite{zar_degree_thm}. Recall that~$c=m-m_0$ is the number of distinguished cycles not at~$p$.

\begin{proposition}[Homotopy type of the down-link]\label{dlk_conn}
 $P(\Gamma)$ is homotopy equivalent to a (possibly empty) wedge of spheres of dimension~$V-c-2$.
\end{proposition}

\begin{proof}
 The proof is similar to the proof of Proposition~2.2 in \cite{vogtmann90} and Proposition~3.2 in \cite{zar_degree_thm}. We induct on the number of admissible edges~$E_{ad}$. Since undistinguished loops do not affect~$P(\Gamma)$,~$V$ or~$c$, we may assume there are none. The base case is~$E_{ad}=0$, for which clearly~$P(\Gamma)$ is empty, i.e.,~$S^{-1}$. When~$m>0$, if there are no admissible edges then~$V=m$ and~$c=m-1$. If~$m=0$ and there are no admissible edges then~$V=1$ and~$c=0$. In both cases, $-1=V-c-2$, which finishes the base case.

 Now assume $E_{ad}>0$, so in particular there exists a maximally distant edge. Let~$\epsilon$ be an optimal (maximally distant) edge. First suppose that~$\epsilon$ is distinguished. By Lemma~\ref{P0_dstgd_edge} and Proposition~\ref{P1_P0}, $P_1(\Gamma,\epsilon)$ is contractible. If~$\epsilon$ is bad then $P(\Gamma)=P_1(\Gamma,\epsilon)$ and we are done, so assume~$\epsilon$ is good. Then $\lk(\epsilon)\cong P(\Gamma/\epsilon)$ by Lemma~\ref{forest_mod_edge}, and admissible blow-downs necessarily decrease~$E_{ad}$, so by induction $\lk(\epsilon)$ is $(V-c-3)$-spherical. Since
 \begin{align*}
  P(\Gamma)=&P_1(\Gamma,\epsilon)\cup\st(\epsilon)\\
 \text{and }&P_1(\Gamma,\epsilon)\cap\st(\epsilon)=\lk(\epsilon)\text{,}
 \end{align*}
 we conclude that $P(\Gamma)$ is $(V-c-2)$-spherical.

 Next suppose that~$\epsilon$ is not distinguished, and is not a separating edge. By the same argument as above, if~$\epsilon$ is good then $\lk(\epsilon)$ is $(V-c-3)$-spherical, so we just need to inspect~$P_1(\Gamma,\epsilon)$, which by Proposition~\ref{P1_P0} is homotopy equivalent to $P_0(\Gamma,\epsilon)$. Since~$\epsilon$ is not a separating edge, we can remove it from~$\Gamma$ and we still have a connected graph with~$m$ distinguished cycles and~$V$ vertices, and strictly fewer admissible edges. By induction then, $P(\Gamma\setminus\epsilon)$ is $(V-c-2)$-spherical (since~$c$ did not change either). Consider the map
 $$g \colon P(\Gamma\setminus\epsilon) \to P_0(\Gamma,\epsilon)$$
 induced by $\Gamma\setminus\epsilon\hookrightarrow\Gamma$. Adding~$\epsilon$ to the graph cannot affect whether a forest~$F$ in $\Gamma\setminus\epsilon$ is admissible or not. Also, since~$\epsilon$ is maximally distant,~$\epsilon$ cannot be decisive, so adding~$\epsilon$ to the graph does not change the levels~$\Lambda_i$. In particular adding~$\epsilon$ cannot affect whether a forest~$F$ in $\Gamma\setminus\epsilon$ is good or bad, so~$g$ is an isomorphism. We conclude that $P_0(\Gamma,\epsilon)$ is $(V-c-2)$-spherical, and hence so is~$P(\Gamma)$. Of course if~$\epsilon$ is bad then $P(\Gamma)=P_1(\Gamma,\epsilon)$, and again we get the result.

 Lastly suppose~$\epsilon$ is not distinguished, but is an (admissible) separating edge. If~$\epsilon$ is good then for any $F\in P(\Gamma)$ it is clear that $F\cup\epsilon$ is again an admissible good forest. In this case~$P(\Gamma)$ is contractible by \cite[Section~1.5]{quillen78}. Incidentally, this completely finishes the~$m=0$ case. If~$\epsilon$ is bad then its top must be a base vertex. Since~$\epsilon$ is maximally distant, and~$\Gamma$ has no undistinguished loops,~$\epsilon$ is the stick of a lollipop $\ell$. The graph $\Gamma\setminus\ell$ has $V-1$ vertices and~$c-1$ distinguished cycles not at~$p$, and has fewer admissible edges than~$\Gamma$. By induction then,
 $$P(\Gamma)=P_1(\Gamma,\epsilon)\simeq P_0(\Gamma,\epsilon)=P(\Gamma\setminus\ell)$$
 is $(V-1)-(c-1)-2=(V-c-2)$-spherical.
\end{proof}

\begin{lemma}[Decisive edges]\label{decisive_edge}
 If~$\Gamma$ has a non-base vertex with an admissible decisive edge then~$P(\Gamma)$ is contractible.
\end{lemma}

\begin{proof}
 The proof is essentially the same as the previous lemma. Induct on~$E_{ad}$. In the base case, there are no admissible edges, much less admissible decisive edges, so the claim is vacuously true. Now assume~$E_{ad}>0$. Let~$\epsilon$ be an optimal maximally distant edge, so~$P_1(\Gamma,\epsilon)$ and~$P_0(\Gamma,\epsilon)$ are homotopy equivalent. If~$\epsilon$ is a separating edge, and good, then~$P(\Gamma)$ is already contractible with cone point~$\epsilon$. If~$\epsilon$ is a separating edge, and bad, then its top is a base vertex. The only way a maximally distant edge can be decisive is if it is separating, and so we can assume there is a decisive edge~$\eta\neq\epsilon$ with top a non-base vertex.

 First suppose that~$\epsilon$ is distinguished. Then~$P_1(\Gamma,\epsilon)$ is contractible, so if~$\epsilon$ is bad we are done. If~$\epsilon$ is good, we still have that $\lk(\epsilon)\cong P(\Gamma/\epsilon)$ as in the previous proof. By Lemma~\ref{which_forests_desc},~$\epsilon$ is either vertical, or is horizontal and connects base vertices. In either case,~$\eta$ maps to a decisive edge in~$\Gamma/\epsilon$, with a non-base vertex for a top, and so~$\lk(\epsilon)$ is contractible by induction. Therefore~$P(\Gamma)$ is contractible. Now suppose~$\epsilon$ is not distinguished. Again,~$\lk(\epsilon)$ is contractible if~$\epsilon$ is good, so we just need to inspect $P_0(\Gamma,\epsilon)$. If~$\epsilon$ is not a separating edge we may remove it as in the previous proof and get that $P_0(\Gamma,\epsilon)\cong P(\Gamma\setminus\epsilon)$ is contractible by induction. The only case remaining is when~$\epsilon$ is a separating edge whose top is a distinguished vertex, so it is the stick of a lollipop $\ell$. Obviously $\eta$ is still a decisive edge in $\Gamma\setminus\ell$, so $P(\Gamma)=P_0(\Gamma,\epsilon)=P(\Gamma\setminus\ell)$ is contractible by induction.
\end{proof}

\subsection{Connectivity of the descending up-link}\label{sec:up_conn}

Now consider the up-link of~$\Gamma$. We return to only considering graphs coming from~$\Delta K_n^m$, so all vertices $v\neq p$ are at least trivalent and~$p$ is at least bivalent. Let~$\BU(v)$ be the poset of all blow-ups at the vertex~$v$, and let~$\DBU(v)$ be the poset of descending blow-ups at~$v$. We will use the combinatorial framework for graph blow-ups described in \cite{culler86} and \cite{vogtmann90}, so we think of~$\BU(v)$ as the poset of \emph{compatible partitions} of the set of incident half-edges. Let $[n]\defeq \{1,\dots n\}$, and consider partitions of~$[n]$ into two blocks. Denote such a partition by~$\alpha=\{a,\bar{a}\}$, where~$1\in a$. Distinct partitions $\{a,\bar{a}\}$ and $\{b,\bar{b}\}$ are called \emph{compatible} if either~$a\subset b$ or~$b\subset a$. Let~$\Sigma(v)$ be the simplicial complex of partitions $\alpha=\{a,\bar{a}\}$ of $[val(v)]$ into blocks~$a$ and~$\bar{a}$ such that~$a$ and~$\bar{a}$ each have at least two elements. (If~$v$ is the basepoint~$p$, then one block may have size one, since~$p$ is allowed to be bivalent.) That is, the vertices of~$\Sigma(v)$ are partitions, and a~$j$-simplex is given by a collection of~$j+1$ distinct, pairwise compatible partitions. Also let~$\Sid(v)$ be the subcomplex of~$\Sigma(v)$ spanned by descending partitions, i.e., partitions that correspond to descending single-edge blow-ups.

For $v\neq p$, the geometric realization $|\!\BU(v)|$ of~$\BU(v)$ is isomorphic to the barycentric subdivision of~$\Sigma(v)$. The idea is that a partition describes an \emph{ideal edge}, i.e., an edge blow-up at a vertex, and the blocks~$a$ and~$\bar{a}$ indicate which half-edges attach to which endpoints of the new edge. See \cite{culler86} and \cite{vogtmann90} for more details. It is also clear that the geometric realization $|\!\DBU(v)|$ contains the barycentric subdivision of~$\Sid(v)$ as a subcomplex, and that any simplex in $|\!\DBU(v)|$ has at least one vertex in~$\Sid(v)$. Hence there is a map $|\!\DBU(v)|\to|\!\DBU(v)|$ sending each simplex to its face spanned by vertices in~$\Sid(v)$, which induces a deformation retraction from $|\!\DBU(v)|$ to~$\Sid(v)$.

The next lemma relates the up-link of~$\Gamma$ to these complexes~$\Sid(v)$. The proof is very similar to the proof of \cite[Proposition~4.5]{zar_degree_thm}.

\begin{lemma}[Local to global]\label{global_blowups}
 Let $\displaystyle\DBU(\Gamma)\defeq \ast_{v\in\Gamma}\DBU(v)$. Then $|\!\DBU(\Gamma)|$ is homotopy equivalent to the up-link of~$\Gamma$.
\end{lemma}

\begin{proof}
 For a poset~$P$, let $\underline{P}$ be $P\sqcup\{\perp\}$, where $\perp$ is a formal minimal element. Then we have that $P\ast Q\simeq \underline{P}\times\underline{Q} - \{(\perp,\perp)\}$. Let
 $$U\defeq \{f\in\prod_v\underline{\BU}(v)-\{(\perp)_v\}\mid f\textnormal{ is descending}\},$$
 so the geometric realization $|U|$ is the up-link. Define a poset map $r:U\to U$ via
 \begin{align*}
  (f_v)_v\mapsto\left(\left\{
  \begin{matrix}
   f_v & \textnormal{ for } & f_v\in\DBU(v)\\
   \perp & \textnormal{ for } & f_v\not\in\DBU(v)
  \end{matrix}
  \right.\right)_v
 \end{align*}
 where $f_v$ is a blow-up at~$v$ in the tuple~$f$. This map is well defined since if~$f$ is descending then $f_v$ must be descending for some~$v$. It is easy to see that $r$ is the identity when restricted to $\DBU(\Gamma)$. Also, $r(f)\le f$ for all $f\in U$, and so by \cite[1.3]{quillen78} this induces a homotopy equivalence between $|U|$ and $|\!\DBU(\Gamma)|$.
\end{proof}

In particular the up-link is homotopy equivalent to $\ast_{v\in\Gamma}\Sid(v)$, so we can analyze the up-link by looking at the complexes~$\Sid(v)$. In light of Lemma~\ref{decisive_edge}, one important situation is when~$v$ is a non-base vertex with no decisive edges.

\begin{lemma}[No decisive edges, locally]\label{local_no_udes}
 Suppose~$v$ is a non-base vertex with no decisive edge. Then $\Sid(v)\simeq\bigvee S^{val(v)-4}$.
\end{lemma}

\begin{proof}
 We know that among the half-edges at~$v$, at least two correspond to vertical edges with top~$v$. Since~$v$ is a non-base vertex, a blow-up at~$v$ is descending if and only if it separates some of these half-edges with top~$v$. (Here the essential term will be $n_{d(p,v)}$.) Thus~$\Sid(v)$ is isomorphic to the complex denoted $\SBU(v)$ in \cite{zar_degree_thm}, and the result is immediate from Lemma~4.1 and Proposition~4.3 in \cite{zar_degree_thm}.
\end{proof}

Next we describe one very important case for which the up-link, and hence~$\dlk(\Gamma)$ is contractible. If a vertex $v\neq p$ has valency 3, or if $v=p$ and $val(v)=2$, we say~$v$ has \emph{minimal valency}. Otherwise we naturally say it has \emph{non-minimal valency}.

\begin{lemma}[Contractible case]\label{fat_base_vtx}
 If~$\Gamma$ has a base vertex with non-minimal valency, then the up-link is contractible, and so~$\dlk(\Gamma)$ is contractible.
\end{lemma}

\begin{proof}
 Let~$v$ be a base vertex with non-minimal valency. By Lemma~\ref{global_blowups} it suffices to show that~$\Sid(v)$ is contractible. Label the distinguished half-edges at~$v$ by $c_1$ and $c_2$, and label the undistinguished half-edges by $b_1,\dots,b_q$. By hypothesis $q>1$, unless $v=p$ in which case $q>0$. Let $\alpha_0$ be the ideal edge at~$v$ that separates $c_1,c_2$ from all the other half-edges. See Figure~\ref{fig:fat_base_vtx} for an example. This is clearly a descending blow-up, with essential term~$m_{d(p,v)}$. Also, any partition of $\{c_1,c_2,b_1,\dots,b_q\}$ that separates $c_1$ and $c_2$ is ascending, so indeed~$\Sid(v)$ is contractible with cone point $\alpha_0$.
\end{proof}

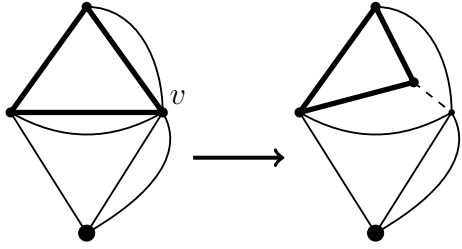
\begin{figure}[htb]
    \centering
    \begin{tikzpicture}%``fat'' base vertex blow-up
  \draw[line width=0.7pt]
   (-1,1.6) -- (0,0) -- (1,1.6)
   (0,0) to [out=30, in=-60,looseness=1] (1,1.6)
   (-1,1.6) to [out=-30, in=210,looseness=1] (1,1.6)
   (0,3) to [out=0, in=90,looseness=1] (1,1.6);
  \draw[line width=2pt]
   (-1,1.6) -- (0,3) -- (1,1.6) -- (-1,1.6);
  \filldraw
  (0,0) circle (3pt)
  (-1,1.6) circle (1.7pt)
  (1,1.6) circle (1.7pt)
  (0,3) circle (1.7pt);
  \node at (1.2,1.8) {$v$};
  \draw[line width=1.5pt, ->]
   (1.4,1) -> (2.6,1);

  \begin{scope}[xshift=3.8cm]
  \draw[line width=0.7pt]
   (-1,1.6) -- (0,0) -- (1,1.6)
   (0,0) to [out=30, in=-60,looseness=1] (1,1.6)
   (-1,1.6) to [out=-30, in=210,looseness=1] (1,1.6)
   (0,3) to [out=0, in=90,looseness=1] (1,1.6);
  \draw[line width=2pt]
   (-1,1.6) -- (0,3) -- (0.5,2) -- (-1,1.6);
  \draw[line width=0.7pt, dashed]
   (0.5,2) -- (1,1.6);
  \filldraw
  (0,0) circle (3pt)
  (-1,1.6) circle (1.7pt)
  (1,1.6) circle (1pt)
  (0.5,2) circle (1.7pt)
  (0,3) circle (1.7pt);
  \end{scope}

\end{tikzpicture}
    \caption{The blow-up at~$v$ given by $\alpha_0$. Here $m_1$ goes from $2$ to~$1$.}
    \label{fig:fat_base_vtx}
\end{figure}

We may now assume every base vertex has minimal valency, and so~$\Sid(v)$ is empty for all base vertices~$v$. Let~$V$ be the number of vertices of~$\Gamma$, and recall that here $d_0=d_0(\Gamma)$ is the degree of~$\Gamma$, i.e., $d_0=2n+2m-val(p)$.

\begin{lemma}[No decisive edges, globally]\label{ulk_relevant_case}
 Suppose~$\Gamma$ has no non-base vertices with an admissible decisive edge. Moreover suppose every base vertex has minimal valency. Then the up-link of~$\Gamma$ is homotopy equivalent to $\bigvee S^{d_0-V}$.
\end{lemma}

\begin{proof}
 By Lemma~\ref{global_blowups}, the up-link is homotopy equivalent to $\ast_{v\in\Gamma}\Sid(v)$. It is clear that $\Sid(p)=\emptyset$, so this is the same as $\ast_{v\neq p}\Sid(v)$. Also, each base vertex $u\neq p$ has valency $3$, so $\Sid(u)=\emptyset=S^{val(u)-4}$. Therefore by Lemma~\ref{local_no_udes} the up-link is homotopy equivalent to
 $$\ast_{v\neq p}(\bigvee S^{val(v)-4}),$$
 which is a wedge of spheres of dimension $\displaystyle (V-2)+\sum_{v\neq p}(val(v)-4)$. Observe that
 $$\sum_{v\neq p}(val(v)-2)=d_0,$$
 so this dimension equals $(V-2)+d_0-2(V-1)=d_0-V$.
\end{proof}

We can now prove our main result of this section. Here $d_w$ is the weighted degree, which recall equals $d_0-c$.

\begin{corollary}[Connectivity of descending links]\label{desc_lk_conn}
 The descending link~$\dlk(\Gamma)$ is either contractible or a wedge of spheres of dimension $d_w-1$.
\end{corollary}

\begin{proof}
 Assume that neither the up-link nor down-link is contractible. Then every base vertex has minimal valency, and no non-base vertex of~$\Gamma$ has a decisive edge. By Proposition~\ref{dlk_conn}, $P(\Gamma)\simeq\bigvee S^{V-c-2}$, and by Lemma~\ref{ulk_relevant_case} the up-link is homotopy equivalent to~$\bigvee S^{d_0-V}$. Hence~$\dlk(\Gamma)$ is homotopy equivalent to
 $$\left(\bigvee S^{V-c-2}\right)\ast\left(\bigvee S^{d_0-V}\right)\simeq\bigvee S^{V-c-2+d_0-V+1}=\bigvee S^{d_0-c-1}=\bigvee S^{d_w-1}.$$
\end{proof}

\subsection{Connectivity of sublevel sets}\label{sec:sublevel_conn}

Using Lemma~\ref{fat_base_vtx} and Corollary~\ref{desc_lk_conn}, we can now finally prove that the sublevel sets $\nabla K_{n,k}^m$ are highly connected, generalizing Hatcher and Vogtmann's Degree Theorem. Recall that the weighted degree~$d_w$ of a graph can never exceed $N=2n+m-1$. Moreover, $d_w=N$ if and only if the basepoint~$p$ has minimal valency and is a base vertex.

\begin{theorem}[Generalized Degree Theorem]\label{thrm:sublev_conn}
 For each~$0\le k < N$, $\nabla K_{n,k}^m$ is $(k-1)$-connected.
\end{theorem}

\begin{proof}
 Since~$\Delta K_n^m$ is contractible, it suffices by \cite[Corollary~2.6]{bestvina97} to show that for any vertex~$\Gamma$ in $\Delta K_n^m \setminus \nabla K_{n,k}^m$, the descending link~$\dlk(\Gamma)$ is at least~$(k-1)$-connected. Let~$\Gamma$ be such a vertex, so either $d_w(\Gamma)>k$, or else $d_w(\Gamma) \le k$ and~$m_0(\Gamma)=1$. In the former case,~$\dlk(\Gamma)$ is $(k-1)$-connected by Corollary~\ref{desc_lk_conn}. In the latter case, the basepoint~$p$ is a base vertex, and since $d_w(\Gamma)<N$,~$p$ has non-minimal valency, so by Lemma~\ref{fat_base_vtx},~$\dlk(\Gamma)$ is contractible.
\end{proof}

\begin{remark}[Concluding remarks]\label{concl_rmk}
 We conclude with some questions that now naturally arise. First, the stable rational homology of $\SAut_n^0$ in~$n$ is trivial, and the rational homology of $\SAut_0^m$ is trivial in every dimension, so it seems likely that the stable homology in~$m$ and~$n$ is always trivial; is this indeed the case? Some additional evidence for this is Theorem~7.4 in \cite{jensen04}, which implies that $H_1(P\SAut_n^m;\Q)=0$ for any $n>2$ and any $m\ge 0$. Second, there exist examples where $H_i(\SAut_n^0;\Q)=\Q$, but when can non-trivial rational homology occur in general, e.g., if $m>0$? This is an interesting question for outer automorphisms as well. Third, when $n=0$ or~$m=0$, we have stable integral homology, so an obvious question is whether this holds in general.
\end{remark}

% --------------------------------------------------------------------

\bibliographystyle{alpha}
\bibliography{bibdata}

\begin{thebibliography}{JMM06}

\bibitem[BB97]{bestvina97}
M.~Bestvina and N.~Brady.
\newblock Morse theory and finiteness properties of groups.
\newblock {\em Invent. Math.}, 129(3):445--470, 1997.

\bibitem[BCV09]{bux09}
K.-U. Bux, R.~Charney, and K.~Vogtmann.
\newblock Automorphisms of two-dimensional {RAAGS} and partially symmetric
  automorphisms of free groups.
\newblock {\em Groups Geom. Dyn.}, 3(4):541--554, 2009.

\bibitem[Bro82]{brown82}
K.~S. Brown.
\newblock {\em Cohomology of groups}, volume~87 of {\em Graduate Texts in
  Mathematics}.
\newblock Springer-Verlag, New York, 1982.

\bibitem[Bux99]{bux99}
K.-U. Bux.
\newblock Orbit spaces of subgroup complexes, {M}orse theory, and a new proof
  of a conjecture of {W}ebb.
\newblock In {\em Proceedings of the 1999 {T}opology and {D}ynamics
  {C}onference ({S}alt {L}ake {C}ity, {UT})}, volume~24, pages 39--51, 1999.

\bibitem[Col89]{collins89}
Donald~J. Collins.
\newblock Cohomological dimension and symmetric automorphisms of a free group.
\newblock {\em Comment. Math. Helv.}, 64(1):44--61, 1989.

\bibitem[CV86]{culler86}
M.~Culler and K.~Vogtmann.
\newblock Moduli of graphs and automorphisms of free groups.
\newblock {\em Invent. Math.}, 84(1):91--119, 1986.

\bibitem[Gal11]{galatius11}
S.~Galatius.
\newblock Stable homology of automorphism groups of free groups.
\newblock {\em Ann. of Math. (2)}, 173(2):705--768, 2011.

\bibitem[Gri12]{griffin12}
J.~Griffin.
\newblock Diagonal complexes and the integral homology of the automorphism
  group of a free product.
\newblock {\em Proc. London Math. Soc.}, 2012.
\newblock To appear. arXiv:1011.6038.

\bibitem[HV98a]{hatcher98}
A.~Hatcher and K.~Vogtmann.
\newblock Cerf theory for graphs.
\newblock {\em J. London Math. Soc. (2)}, 58(3):633--655, 1998.

\bibitem[HV98b]{hatcher98b}
A.~Hatcher and K.~Vogtmann.
\newblock Rational homology of {${\rm Aut}(F\_n)$}.
\newblock {\em Math. Res. Lett.}, 5(6):759--780, 1998.

\bibitem[HW05]{hatcher05}
A.~Hatcher and N.~Wahl.
\newblock Stabilization for the automorphisms of free groups with boundaries.
\newblock {\em Geom. Topol.}, 9:1295--1336 (electronic), 2005.

\bibitem[HW10]{hatcher10}
A.~Hatcher and N.~Wahl.
\newblock Stabilization for mapping class groups of 3-manifolds.
\newblock {\em Duke Math. J.}, 155(2):205--269, 2010.

\bibitem[JMM06]{jensen06}
C.~A. Jensen, J.~McCammond, and J.~Meier.
\newblock The integral cohomology of the group of loops.
\newblock {\em Geom. Topol.}, 10:759--784, 2006.

\bibitem[JW04]{jensen04}
C.~A. Jensen and N.~Wahl.
\newblock Automorphisms of free groups with boundaries.
\newblock {\em Algebr. Geom. Topol.}, 4:543--569, 2004.

\bibitem[McE10]{mcewen_thesis}
R.~A. McEwen.
\newblock {\em Homological stability for the groups {O}ut{P}(n,t+1)}.
\newblock PhD thesis, 2010.
\newblock Thesis (Ph.D.)--University of Virginia.

\bibitem[MZ12]{zar_degree_thm}
R.~McEwen and M.~C.~B. Zaremsky.
\newblock A combinatorial proof of the degree theorem in auter space.
\newblock 2012.
\newblock arXiv:0907.4642. Submitted.

\bibitem[Qui78]{quillen78}
D.~Quillen.
\newblock Homotopy properties of the poset of nontrivial {$p$}-subgroups of a
  group.
\newblock {\em Adv. in Math.}, 28(2):101--128, 1978.

\bibitem[Vog90]{vogtmann90}
Karen Vogtmann.
\newblock Local structure of some {${\rm Out}(F\_n)$}-complexes.
\newblock {\em Proc. Edinburgh Math. Soc. (2)}, 33(3):367--379, 1990.

\bibitem[Wil12]{wilson12}
J.~C.~H. Wilson.
\newblock Representation stability for the cohomology of the pure string motion
  groups.
\newblock {\em Algebr. Geom. Topol.}, 12(2):909--931, 2012.

\end{thebibliography}

% --------------------------------------------------------------------

\end{document}